\documentclass[letterpaper,11pt]{amsart}

\usepackage[margin=1in]{geometry}
\usepackage{amsmath,amsthm,amssymb, mathtools}
\usepackage{xspace,xcolor}
\usepackage[breaklinks,colorlinks,citecolor=teal,linkcolor=teal,urlcolor=teal,pagebackref,hyperindex]{hyperref}
\usepackage[alphabetic]{amsrefs}
\usepackage{comment}
\usepackage{enumerate}
\usepackage[all]{xy}
\usepackage{tikz-cd}
\usepackage{color}
\usepackage{comment}
\usepackage{nccmath}

\theoremstyle{plain}
\newtheorem{theo}{Theorem}[section]

\newtheorem{lemm}[theo]{Lemma}
\newtheorem{prop}[theo]{Proposition}
\newtheorem{coro}[theo]{Corollary}
\newtheorem{ques}[theo]{Question}

\theoremstyle{definition}
\newtheorem{defi}[theo]{Definition}

\newtheorem{exam}[theo]{Example}

\theoremstyle{remark}
\newtheorem{rema}[theo]{Remark}

\newcommand{\eps}{\epsilon}

\newcommand{\bb}[1]{\mathbb{#1}}

\newcommand{\PP}{\mathbb{P}}
\newcommand{\QQ}{\mathbf{Q}}

\newcommand{\RR}{\mathbf{R}}
\newcommand{\CC}{\mathbf{C}}

\newcommand{\DD}{\mathbb{D}}

\newcommand{\ZZ}{\mathbf{Z}}

\newcommand{\cF}{\mathcal{F}}

\newcommand{\cH}{\mathcal{H}}

\newcommand{\cK}{\mathcal{K}}

\newcommand{\cM}{\mathcal{M}}

\newcommand{\cO}{\mathcal{O}}

\newcommand{\cT}{\mathcal{T}}
\newcommand{\cX}{\mathcal{X}}

\newcommand{\uind}[1]{^{(#1)}}

\newcommand{\lsta}{_{*}}

\newcommand{\sta}{^{*}}

\newcommand{\dual}{^{\vee}}

\newcommand{\norm}[1]{\|#1\|}

\newcommand{\uround}[1]{\left\lceil #1 \right\rceil }
\newcommand{\lround}[1]{\left\lfloor #1 \right\rfloor}

\DeclareMathOperator{\Hom}{Hom}

\DeclareMathOperator{\Aut}{Aut}

\DeclareMathOperator{\Ext}{Ext}

\DeclareMathOperator{\im}{im}

\DeclareMathOperator{\id}{id}

\DeclareMathOperator{\gr}{gr}

\DeclareMathOperator{\Spec}{Spec}

\DeclareMathOperator{\vol}{vol}
\DeclareMathOperator{\ord}{ord}

\newcommand{\de}{\partial}

\DeclareMathOperator{\prim}{prim}

\DeclareMathOperator{\IC}{IC}
\DeclareMathOperator{\IH}{IH}

\DeclareMathOperator{\MHM}{MHM}

\newcommand{\bft}{\mathbf{t}}

\usepackage{soul}

\begin{document}
	\thanks{The authors were partially supported by the Simons Collaboration grant \textit{Moduli of Varieties}.}
	
	
	\author{Hyunsuk Kim}
	
	\address{Department of Mathematics, University of Michigan, 530 Church Street, Ann Arbor, MI 48109, USA}
	
	\email{kimhysuk@umich.edu}	
	
	\title[Hodge theory and $K$-stability of some very symmetric hypersurfaces]{Hodge theory and $K$-stability of some very symmetric hypersurfaces}
	
	\maketitle
    \begin{abstract}
        We study some interesting hypersurfaces that naturally arise when studying the period map on the moduli space of hypersurfaces, in the context of Sung Gi Park's recent work on studying the GIT moduli space of hypersurfaces via the minimal exponent. We compute the Hodge structure on the singular cohomology and the intersection cohomology of these hypersurfaces, and also show the $K$-polystability of certain mildly singular degenerate hypersurfaces among them. In particular, the following hypersurface is $K$-polystable for $l \geq 2$:
        $$ \{ x_{11}\cdots x_{1d} + \ldots + x_{ld} \cdots x_{ld} = 0\} \subset \PP^{ld-1}.$$
	\end{abstract}
	

    \section{Introduction}    
    The main purpose of this article is to report some interesting hypersurfaces that naturally arise when studying the period map on the GIT moduli space of hypersurfaces.
    
    Let $X$ be a hypersurface in $\PP^{n+1}$. If $X$ is smooth, the Hodge structure on the singular cohomology of $X$ is classically known \cite{Carlson--Griffiths}. If $X$ is singular, computing the singular cohomology of $X$ turns out to be a non-trivial task; for relevant literature, see \cite{Dimca}*{\S6},\cite{Dimca-Saito}, \cite{Dimca-Saito-Wotzlaw}. This is an important problem if one wants to study the period map defined on compactified moduli spaces of smooth hypersurfaces to nice compactifications of period domains, for example, \cite{Shah-degree-2-K3,Shah-degree-4-K3,Engel-Alexeev,Allcock-Carlson-Toledo,Laza:4-fold-period, Looijenga-cubic-fourfold} to list a few among many interesting results. In \cite{park2025git}, Sung Gi Park provides a new perspective on the aforementioned results in terms of Hodge modules and the theory of higher Du Bois and rational singularities. In particular, he discovers a tight relation between the GIT stability of a hypersurface $X$ in $\PP^{n}$ and its \textit{minimal exponent} $\widetilde{\alpha}(X)$, an invariant of singularities detecting higher Du Bois and rational singularities. More precisely, he shows that a degree $d$ hypersurface $X$ in $\PP^{n}$ is GIT stable (resp.\ semi-stable) if $\widetilde{\alpha}(X) > \frac{n+1}{d}$ (resp.\ $\geq \frac{n+1}{d}$). If $\frac{n+1}{d}$ is an integer, the relation between higher Du Bois and rational singularities becomes more transparent.
    
    The following hypersurfaces are important examples arising in this picture:
    $$ Y = \{ f(x_{0},\ldots, x_{m+1}) + y_{11}\cdots y_{1d} + \ldots + y_{l1} \cdots y_{ld} = 0 \} \subset \PP^{m+1+ld}, $$
    where $X = \{ f(x) = 0 \} \subset \PP^{m+1}$ is a smooth hypersurface (see \cite{park2025git}*{Example 13.6}). If $d$ divides $m+2$, the important part (the smallest Hodge subquotient containing the lowest piece of the Hodge filtration) of the limiting mixed Hodge structure of any one parameter degeneration to $Y$ agrees with that of the singular cohomology of the hypersurface $X$, up to a Tate twist. We refer to \textit{loc.cit.} for a precise formulation of this result.

    The most degenerate case is given by the equation
    \begin{equation} \label{eqn:X_ld}
        X_{l,d} = \{ x_{11}\cdots x_{1d} + \ldots + x_{l1} \cdots x_{ld} = 0 \} \subset \PP^{ld-1}.
    \end{equation}
    Note that the Hodge structure of the middle cohomology of a smooth hypersurface $Y$ of degree $d$ in $\PP^{ld-1}$ is of `Calabi--Yau type', i.e., 
    $$ h^{l-1, ld-l-1}(Y) = h^{ld-l-1,l-1}(Y) = 1, \quad \text{and}\quad h^{i, ld-2-i}(Y) = h^{ld-2-i,i}(Y) = 0 \text{ for } 0 \leq i < l-1.$$
    The Hodge-theoretic behavior of this degeneration is exactly the same as the maximally unipotent degeneration of Calabi--Yau hypersurfaces to the toric boundary, for example:
    $$ \cX = \left\{ ([x_1:\ldots :x_d], t) : x_{1}\cdots x_{d} + t (x_{1}^{d} + \ldots + x_{d}^{d}) = 0 \right\} \subset \PP_{\DD}^{d-1}.$$
    For example, any one parameter smoothing $\cX \to \DD$ of $X_{l,d}$ is maximally unipotent (see \cite{park2025git}*{Corollary G, Example 13.5}). In some sense, one can understand the hypersurface $X_{l,d}$ as the most degenerate one among degree $d$ hypersurfaces with $(l-1)$-Du Bois singularities. This is the case when $(l,d) = (2,3)$ according to Laza's classification of GIT polystable cubic fourfolds \cite{Laza-fourfold,Laza:4-fold-period}. Indeed, this hypersurface is the intersection of all the boundary components of the GIT moduli space of cubic fourfolds. From this point of view, it is natural to ask about various Hodge theoretic invariants of these hypersurfaces.

    We introduce a notation. For a hypersurface $X \subset \PP^{N+1}$, we will consider the singular cohomology modulo the hyperplane classes:
    $$ \overline{H}^{i}(X, \QQ) \coloneqq \begin{cases} H^{i}(X, \QQ)/ \langle h^{i/2} \rangle & i \text{ is even}\\H^{i}(X, \QQ) & i \text{ is odd}.
    \end{cases}$$
    Here, $h \in H^{2}(X, \QQ)$ is the restriction of the hyperplane class from the projective space. 

    \begin{theo} \label{theo:appending-snc}
        Let $X \subset \PP^{m+1}$ be a hypersurface given by the homogeneous polynomial $f(x)$ of degree $d$. Consider the hypersurface
        $$ Y = \{ f(x) + y_{11} \cdots y_{1d} + \ldots + y_{l1}\cdots y_{ld} = 0\} \subset \PP^{m+ 1+ld}.$$
        The singular cohomology of $X$ and $Y$ is related in the following manner:
        $$ \overline{H}^{m+ld+ i}(Y) \simeq \bigoplus_{0 \leq i_{0}} \overline{H}^{m + i_{0}}(X) (-l-i+i_{0})^{\oplus a_{i-i_{0}}}.$$
        In particular, if $X$ is smooth, we have
        $$ \overline{H}^{m+ld + i}(Y) \simeq H^{m}_{\prim}(X)(-l-i)^{\oplus a_{i}}.$$
        Here, the multiplicities $a_{i}$ are defined as
        $$ a_{i} = \sum_{\substack{i_{1}+ \ldots + i_{l} = i \\ 0 \leq i_{j} \leq d-2}} {d - 1 \choose i_{1}}\cdots {d - 1 \choose i_{l}}.$$
        These numbers are non-zero for $0 \leq i \leq l(d-2)$.
    \end{theo}

    \begin{rema}
        We also remark that the short exact sequence of mixed Hodge structures
        $$ 0 \to \QQ \cdot h^{i} \to H^{2i}(X, \QQ) \to \overline{H}^{2i}(X, \QQ) \to 0 $$
        is split since $H^{2i}(X, \QQ)$ has weights $\leq 2i$.
    \end{rema}

    \begin{rema}
        We note that there is no assumption on the singularities of the hypersurface $X$. However, we remark that using the Thom--Sebastiani formula for the minimal exponent (\cite{Saito-microlocal-b-function}*{Theorem 0.8}) and the relation between the minimal exponent and that of its cone (\cite{park2025git}*{Theorem 6.1}), it is easy to check that if $\widetilde{\alpha}(X) \geq \frac{m+2}{d}$, then $\widetilde{\alpha}(Y) \geq \frac{m + 2}{d} + l$. From this point of view, the minimal exponent suggests that appending the variables $y_{11}\cdots y_{1d} + \ldots + y_{l1}\cdots y_{ld}$ provides an interesting morphism from a moduli space of lower dimensional hypersurfaces to a moduli space of higher dimensional ones. For the GIT moduli space, we show that we indeed have a morphism between two moduli spaces.
    \end{rema}

    \begin{prop} \label{prop:appending-snc-preserves-GIT}
        Let $X \subset \PP^{m+1}$ be a hypersurface given by the homogeneous polynomial $f(x)$ of degree $d$. Consider the hypersurface $Y = \{ f(x) + y_{1}\cdots y_{d} = 0\}\subset \PP^{m+d+1}$. If $X$ is GIT semi(resp.\! poly)-stable, then $Y$ is GIT semi(resp.\! poly)-stable.
    \end{prop}

    We also compute the automorphism group and the intersection cohomology of $X_{l,d}$, introduced in (\ref{eqn:X_ld}).

    \begin{theo}\label{theo:main-theorem}
        Let $X_{l,d}$ be the hypersurface in $\PP^{ld-1}$ defined by the equation
        $$ x_{11} \cdots x_{1d} + \ldots + x_{l1}\cdots x_{ld} = 0.$$
        \begin{enumerate}
            \item The automorphism group $\Aut(X_{l,d})$ is a semi-direct product of a finite group and a torus. For the exact description of the group, see Proposition \ref{prop:automorphism-of-X}.
            \item The $\QQ$-Hodge structure of the singular cohomology of $X_{l,d}$ is given by
            $$ \overline{H}^{ld-2+i}(X_{l,d}, \QQ) \simeq \QQ(-(l-1+i))^{\oplus a_{i}}$$
            where $a_{i}$ is defined as in Theorem \ref{theo:appending-snc}.
            \item The intersection cohomology of $X_{l,d}$ is pure of Hodge--Tate type, and we have an explicit formula for the intersection Betti numbers.  We refer to Proposition \ref{prop:intersection-cohomology} for the formula.
        \end{enumerate}
    \end{theo}

    \begin{rema}
        The computation of the singular cohomology for the case $l = 2$ is carried out in \cite{kim:LCDTV-1}*{Example 7.3} using toric methods. 
    \end{rema}

    \begin{exam}[Cubic sevenfolds]
        We demonstrate how Theorem \ref{theo:appending-snc} can be carried out in the case of cubic sevenfolds, by drawing the upper half of the Hodge--Du Bois diamond:
        \begin{center}
\begin{tikzpicture}
\node at (0, 0) {$\underline{h}^{7,7}$}; 
\node at (-0.5, -0.5) {$\underline{h}^{7,6}$}; 
\node at (0.5, -0.5) {$\underline{h}^{6,7}$}; 
\node at (-1, -1) {$\underline{h}^{7,5}$}; 
\node at (0, -1) {$\underline{h}^{6,6}$}; 
\node at (1, -1) {$\underline{h}^{5,7}$}; 
\node at (-1.5, -1.5) {$\underline{h}^{7,4}$}; 
\node at (-0.5,-1.5) {$\underline{h}^{6,5}$};
\node at (0.5,-1.5) {$\underline{h}^{5,6}$};
\node at (1.5, -1.5) {$\underline{h}^{4,7}$}; 
\node at (-2, -2) {$\underline{h}^{7,3}$}; 
\node at (-1, -2) {$\underline{h}^{6,4}$}; 
\node at (0, -2) {$\underline{h}^{5,5}$}; 
\node at (1, -2) {$\underline{h}^{4,6}$}; 
\node at (2, -2) {$\underline{h}^{3,7}$};
\node at (-2.5, -2.5) {$\underline{h}^{7,2}$};
\node at (-1.5, -2.5) {$\underline{h}^{6,3}$};
\node at (-0.5, -2.5) {$\underline{h}^{5,4}$};
\node at (0.5, -2.5) {$\underline{h}^{4,5}$};
\node at (1.5, -2.5) {$\underline{h}^{3,6}$}; 
\node at (2.5, -2.5) {$\underline{h}^{2,7}$}; 
\node at (-3, -3) {$\underline{h}^{7,1}$}; 
\node at (-2, -3) {$\underline{h}^{6,2}$};
\node at (-1, -3) {$\underline{h}^{5,3}$};
\node at (0, -3) {$\underline{h}^{4,4}$};
\node at (1, -3) {$\underline{h}^{3,5}$};
\node at (2, -3) {$\underline{h}^{2,6}$};
\node at (3, -3) {$\underline{h}^{1,7}$};
\node at (-3.5, -3.5) {$\underline{h}^{7,0}$};
\node at (-2.5, -3.5) {$\underline{h}^{6,1}$};
\node at (-1.5, -3.5) {$\underline{h}^{5,2}$};
\node at (-0.5, -3.5) {$\underline{h}^{4,3}$};
\node at (0.5, -3.5) {$\underline{h}^{3,4}$};
\node at (1.5, -3.5) {$\underline{h}^{2,5}$}; 
\node at (2.5, -3.5) {$\underline{h}^{1,6}$}; 
\node at (3.5, -3.5) {$\underline{h}^{0,7}$.}; 
\end{tikzpicture}
\end{center}
The Hodge--Du Bois numbers are defined as
$$ \underline{h}^{p,q}(X) = \dim \gr_{F}^{p} H^{p+q}(X, \CC).$$
We note that the lower half of the Hodge diamond is easily determined by the Lefschetz hyperplane theorem (see \cite{Positivity-Book-I}*{Theorem 3.1.17}).

        Case 1. Suppose $X = \{ f(x_{0},\ldots, x_{8}) = 0\} \subset \PP^{8}$ be a smooth cubic sevenfold. The only interesting singular cohomology of $X$ is $H^{7}$, and the middle row of the Hodge diamond is
        $$ 0 \quad 0 \quad 1 \quad 84 \quad 84 \quad 1 \quad 0 \quad 0.$$

        Case 2. Suppose $X = \{ x_{0}x_{1}x_{2} + f(x_{3},\ldots, x_{8}) = 0\} \subset \PP^{8}$, where $f$ defines a smooth cubic fourfold. Then the Hodge diamond of $X$ is
        \begin{center}
\begin{tikzpicture}
\node at (0, 0) {1}; 
\node at (-0.5, -0.5) {0}; 
\node at (0.5, -0.5) {0}; 
\node at (-1, -1) {0}; 
\node at (0, -1) {1}; 
\node at (1, -1) {0}; 
\node at (-1.5, -1.5) {0}; 
\node at (-0.5,-1.5) {0};
\node at (0.5,-1.5) {0};
\node at (1.5, -1.5) {0}; 
\node at (-2, -2) {0}; 
\node at (-1, -2) {0}; 
\node at (0, -2) {1}; 
\node at (1, -2) {0}; 
\node at (2, -2) {0};
\node at (-2.5, -2.5) {0};
\node at (-1.5, -2.5) {0};
\node at (-0.5, -2.5) {0};
\node at (0.5, -2.5) {0};
\node at (1.5, -2.5) {0}; 
\node at (2.5, -2.5) {0}; 
\node at (-3, -3) {0}; 
\node at (-2, -3) {0};
\node at (-1, -3) {2};
\node at (0, -3) {41};
\node at (1, -3) {2};
\node at (2, -3) {0};
\node at (3, -3) {0};
\node at (-3.5, -3.5) {0};
\node at (-2.5, -3.5) {0};
\node at (-1.5, -3.5) {0};
\node at (-0.5, -3.5) {1};
\node at (0.5, -3.5) {20};
\node at (1.5, -3.5) {1}; 
\node at (2.5, -3.5) {0}; 
\node at (3.5, -3.5) {0.}; 
\end{tikzpicture}
\end{center}
The `1, 20, 1' on the lowest row is the Tate twist of the primitive middle cohomology of the corresponding cubic fourfold. The `2, 41, 2' on the next row should be interpreted as two copies of `1, 20, 1's from the cubic fourfold and an extra Hodge--Tate type Hodge structure of dimension 1.

Case 2. Suppose $X = \{ x_{0}x_{1}x_{2} + x_{3}x_{4}x_{5} + f(x_{6}, x_{7}, x_{8}) = 0\} \subset \PP^{8}$, where $f$ defines a smooth cubic curve. Then the Hodge diamond of $X$ is
\begin{center}
\begin{tikzpicture}
\node at (0, 0) {1}; 
\node at (-0.5, -0.5) {0}; 
\node at (0.5, -0.5) {0}; 
\node at (-1, -1) {0}; 
\node at (0, -1) {1}; 
\node at (1, -1) {0}; 
\node at (-1.5, -1.5) {0}; 
\node at (-0.5,-1.5) {0};
\node at (0.5,-1.5) {0};
\node at (1.5, -1.5) {0}; 
\node at (-2, -2) {0}; 
\node at (-1, -2) {0}; 
\node at (0, -2) {1}; 
\node at (1, -2) {0}; 
\node at (2, -2) {0};
\node at (-2.5, -2.5) {0};
\node at (-1.5, -2.5) {0};
\node at (-0.5, -2.5) {4};
\node at (0.5, -2.5) {4};
\node at (1.5, -2.5) {0}; 
\node at (2.5, -2.5) {0}; 
\node at (-3, -3) {0}; 
\node at (-2, -3) {0};
\node at (-1, -3) {0};
\node at (0, -3) {5};
\node at (1, -3) {4};
\node at (2, -3) {0};
\node at (3, -3) {0};
\node at (-3.5, -3.5) {0};
\node at (-2.5, -3.5) {0};
\node at (-1.5, -3.5) {0};
\node at (-0.5, -3.5) {0};
\node at (0.5, -3.5) {1};
\node at (1.5, -3.5) {1}; 
\node at (2.5, -3.5) {0}; 
\node at (3.5, -3.5) {0.}; 
\end{tikzpicture}
\end{center}
The `1, 1' on the lowest row is the Tate twist of the middle cohomology of the elliptic curve. Then `4, 4' and the `5, 4' on the next two rows shuold be interpreted as four copies of `1, 1's, possibly with an extra Hodge--Tate type Hodge structure.

Case 3. Suppose $X= \{ x_{0}x_{1}x_{2} + x_{3}x_{4}x_{5} + x_{6}x_{7}x_{8} = 0\} \subset \PP^{8}$. Then the Hodge diamond of $X$ is
\begin{center}
\begin{tikzpicture}
\node at (0, 0) {1}; 
\node at (-0.5, -0.5) {0}; 
\node at (0.5, -0.5) {0}; 
\node at (-1, -1) {0}; 
\node at (0, -1) {1}; 
\node at (1, -1) {0}; 
\node at (-1.5, -1.5) {0}; 
\node at (-0.5,-1.5) {0};
\node at (0.5,-1.5) {0};
\node at (1.5, -1.5) {0}; 
\node at (-2, -2) {0}; 
\node at (-1, -2) {0}; 
\node at (0, -2) {9}; 
\node at (1, -2) {0}; 
\node at (2, -2) {0};
\node at (-2.5, -2.5) {0};
\node at (-1.5, -2.5) {0};
\node at (-0.5, -2.5) {0};
\node at (0.5, -2.5) {12};
\node at (1.5, -2.5) {0}; 
\node at (2.5, -2.5) {0}; 
\node at (-3, -3) {0}; 
\node at (-2, -3) {0};
\node at (-1, -3) {0};
\node at (0, -3) {1};
\node at (1, -3) {6};
\node at (2, -3) {0};
\node at (3, -3) {0};
\node at (-3.5, -3.5) {0};
\node at (-2.5, -3.5) {0};
\node at (-1.5, -3.5) {0};
\node at (-0.5, -3.5) {0};
\node at (0.5, -3.5) {0};
\node at (1.5, -3.5) {1}; 
\node at (2.5, -3.5) {0}; 
\node at (3.5, -3.5) {0.}; 
\end{tikzpicture}
\end{center}
All of the Hodge structures are mixed of Hodge--Tate type.
\end{exam}

The notion of $K$-stability arose in order to find Kähler--Einstein metrics on Fano varieties. The algebraic reformulation of $K$-stability turned out to be an extremely useful ingredient for constructing the moduli space of Fano varieties (see \cite{Xu:K-book} and the references therein). If $l \geq 2$, the hypersurfaces $X_{l,d}$ are Fano, and it is natural to ask if they are $K$-(semi, poly)stable. We indeed show that $X_{l,d}$ are $K$-polystable.

\begin{theo} \label{theo:K-polystable}
        The hypersurfaces $X_{l,d}$ as in Theorem \ref{theo:main-theorem} are $K$-polystable for $l \geq 2$.
\end{theo}

Combining it with \cite{Liu-Zhu:equivariant-K-stability-group-fini}*{Theorem 5.1}, we immediately get the following by iteratively taking cyclic covers branched along the hypersurface.
\begin{coro}
    The hypersurfaces
    $$ X = \{ y_{0}^{d} + \ldots + y_{m+1}^{d} + x_{11}\cdots x_{1d} + \ldots + x_{l1} \cdots x_{ld}= 0\} \subset \PP^{ld + m+1}$$
    are $K$-polystable. \hfill{$\Box$}
\end{coro}

We also partially analyze the local geometry of the $K$-moduli space near the hypersurface $X_{l,d}$. By doing that, we also recover a weaker statement of the corollary above; namely, we get that the hypersurfaces
$ X = \{ f(x_{11},\ldots, x_{td}) + x_{t+1,1}\cdots x_{t+1,d} + \ldots + x_{l1} \cdots x_{ld} = 0 \} \subset \PP^{ld-1} $
are $K$-polystable for $f$ general.

We also remark that it is possible to see other interesting hypersurfaces near $[X_{l,d}]$ other than the ones above. The dimension of the singular locus of $X_{l,d}$ is relatively large, and we have various strata near the point $[X_{l,d}]$ whose topological behavior of the hypersurfaces on them are different.

We end the introduction by asking a parallel question to Proposition \ref{prop:appending-snc-preserves-GIT} in the $K$-moduli set-up.

\begin{ques}
    Does the same statement in Proposition \ref{prop:appending-snc-preserves-GIT} hold if we replace GIT stability with $K$-stability?
\end{ques}

    \noindent
    \textbf{Organization of the paper.} In \S2, we collect some preliminary facts. We prove Theorem \ref{theo:appending-snc} in \S3 with its byproduct Theorem \ref{theo:main-theorem}(2). In \S4, we show Theorem \ref{theo:main-theorem}(3). In \S5, we show Theorem \ref{theo:main-theorem}(1), Theorem \ref{theo:K-polystable}, and give a partial analysis of the local geometry of the $K$-moduli space near $X_{l,d}$. In \S6, we prove Proposition \ref{prop:appending-snc-preserves-GIT}.
    
    \section{Preliminaries}
    \subsection{Hodge modules}
    We use the language of mixed Hodge modules \cite{saito1988modulesdeHodge,saito1990mixedHodgemodules} in a very mild manner. Since we do not need the full theory of mixed Hodge modules, throughout the paper, one could understand them as certain constructible complexes with additional data that allow us to keep track of the Hodge structures and weights. A mixed Hodge module has an underlying $\QQ$-complex which is a perverse sheaf. There is a six functor formalism and various functors defined in the category of mixed Hodge modules which are compatible with the corresponding functors defined in the level of constructible complexes.
    
    Over a point, the category of mixed Hodge modules is the abelian category of mixed Hodge structures. Therefore, whenever a natural operation on the level of constructible complexes spits out a vector space, the corresponding operation in the level of mixed Hodge modules outputs a Hodge structure.
    
    Saito's theory provides objects $\QQ_{X}^{H}[\dim X]$ and $\IC_{X}^{H}$ on $X$ whose underlying constructible sheaves are the constant sheaf $\QQ_{X}[\dim X]$ and the intersection complex $\IC_{X}$, respectively. We remark that in general, the object $\QQ_{X}^{H}[\dim X]$ lives in the derived category of mixed Hodge modules. However, throughout this paper, this will always be a genuine mixed Hodge module since $X$ has hypersurface singularities. If we push forward these objects to a point and take cohomology, we obtain the Hodge structures of the singular cohomology and the intersection cohomology of $X$. By \cite{saito1990mixedHodgemodules}, the highest graded piece of the weight filtration is $\gr_{\dim X}^{W}\QQ_{X}^{H}[\dim X] \simeq \IC_{X}^{H}$ (again, we assume that $X$ has hypersurface singularities). In particular, we have the short exact sequence
    $$0 \to \cK_{X} \to \QQ_{X}^{H}[\dim X] \to \IC_{X}^{H} \to 0. $$
    We call $\cK_{X}$ the RHM defect object of $X$ following the terminology in \cite{park-popa:lefschetz-Hodge-symmetry}.

    \subsection{Nearby and vanishing cycles}
    Consider a polynomial function $f \colon \CC^{m+1} \to \CC$ with central fibre $X = f^{-1}(0)_{\mathrm{red}}$. Then the nearby and vanishing cycles provide functors between the category of mixed Hodge modules
    $$ \psi_{f} \colon \MHM(\CC^{m+1}) \to \MHM(X), \qquad \varphi_{f} \colon \MHM(\CC^{m+1}) \to \MHM(X).$$
    This is compatible with the corresponding (shifted) functors constructed in the level of constructible complexes. Throughout this paper, we will only need the nearby and vanishing cycles of the constant mixed Hodge module $\QQ_{\CC^{m+1}}^{H}[m+1]$. Therefore, we will use a slight abuse of notation:
    $$ \psi_{f} \coloneqq \psi_{f} \QQ_{\CC^{m+1}}^{H}[m+1]$$
    and analogously for $\varphi_{f}$. We denote by $\psi_{f,1}$ and $\varphi_{f,1}$ for the unipotent parts of $\psi_{f}$ and $\varphi_{f}$, respectively.

    There are canonical morphisms of mixed Hodge modules
    $$ \mathrm{can} \colon \psi_{f, 1} \to \varphi_{f, 1}, \qquad \mathrm{Var} \colon \varphi_{f, 1} \to \psi_{f,1}(-1). $$
    The compositions $\mathrm{can}(-1) \circ \mathrm{Var} \colon \varphi_{f, 1} \to \varphi_{f,1}(-1)$ and $\mathrm{Var}\circ \mathrm{can} \colon \psi_{f, 1} \to \psi_{f, 1}(-1)$ are the nilpotent operator $N$ induced by the logarithm of the monodromy. The monodromy weight filtration of $N$ agrees with the weight filtration on $\psi_{f,1}$ and $\varphi_{f,1}$ with central weight $m$ and $m+1$, respectively.
    
    The morphism $\mathrm{can}$ is surjective, and $\mathrm{Var}$ is injective. We have a short exact sequence of mixed Hodge modules:
    $$ 0 \to \QQ_{X}^{H}[m] \to \psi_{f,1} \xrightarrow{\mathrm{can}} \varphi_{f,1} \to 0.$$

    We show a small lemma.

    \begin{lemm} \label{lemm:defect-is-kernel-of-N}
        As in the set-up above, we have an isomorphism
        $$ \cK_{X}(-1) \simeq \ker (N \colon \varphi_{f,1} \to \varphi_{f, 1}(-1)).$$
    \end{lemm}
    \begin{proof}
        Temporarily, we will use the notation $N_{\psi}$ and $N_{\varphi}$ for the nilpotent operator acting on $\psi$ and $\varphi$ respectively, in order to avoid confusion. By the surjectivity of $\mathrm{can}$ and the injectivity of $\mathrm{Var}$, the trivial Hodge module $\QQ_{X}^{H}[m]$ is identified with $\ker N_{\psi}$. Since $\IC_{X}^{H} = \gr_{ld-1}^{W} \QQ_{X}^{H}[ld-1]$, we see that
        $$ \cK_{X} \simeq \ker N_{\psi} \cap \im N_{\psi}(1).$$
        We have a short exact sequence
        $$ 0 \to \ker N_{\psi} \hookrightarrow \ker N_{\psi}^{2} \xrightarrow{N_{\psi}}  \ker N_{\psi}(-1) \cap \im N_{\psi} \to 0.$$
        Therefore, it is enough to show that $\mathrm{can}(\ker N_{\psi}^{2}) = \ker N_{\varphi}$. This follows from the description
        $$ N_{\psi}^{2} \colon \psi_{f,1} \xrightarrow{\mathrm{can}} \varphi_{f,1} \xrightarrow{N_{\varphi}} \varphi_{f,1}(-1) \xrightarrow{ \mathrm{Var}(-1)} \psi_{f,1}(-2),$$
        the surjectivity of $\mathrm{can}$ and the injectivity of $\mathrm{Var}$.
    \end{proof}

    We recall a Thom--Sebastiani type theorem for mixed Hodge modules.
    \begin{prop}[\cite{Maxim-Schurmann-Saito:Thom-Sebastiani}*{Theorem 2}] \label{prop:Thom-Sebastiani}
        Let $f_{i} \colon \CC^{m_{i}}\to \CC$ for $1 \leq i \leq l$. Let $f \colon \CC^{m_{1}+ \ldots + m_{l}} \to \CC$ given by
        $$ f(z_1,\ldots, z_{l}) = f_{1}(z_{1}) + \ldots + f_{l}(z_{l}), \qquad z_{1} \in \CC^{m_{1}},\ldots, z_{l} \in \CC^{m_{l}}.$$
        Suppose that the monodromies on $\varphi_{f_{2}},\ldots, \varphi_{f_{l}}$ are unipotent, i.e., $\varphi_{f_{i}} = \varphi_{f_{i},1}$ for $2\leq i \leq l$. Then
        $$ \varphi_{f} \simeq \varphi_{f_{1}} \boxtimes \ldots \boxtimes \varphi_{f_{l}}$$
        as a mixed Hodge module supported on $\CC^{m_{1} + \ldots + m_{l}}$. If we denote $N_{i}$ the nilpotent operator on $\varphi_{f_{i}}$, the nilpotent operator on $\varphi_{f}$ is
        $$ N = \sum_{i=1}^{l} 1 \boxtimes \ldots \boxtimes N_{i} \boxtimes 1 \ldots \boxtimes 1.$$
    \end{prop}
    \begin{rema}
        We note that the statement in \cite{Maxim-Schurmann-Saito:Thom-Sebastiani} also deals with the case when there are various eigenvalues of the monodromy. However, there is a shift of the Hodge filtration and the weight filtration depending on the eigenvalues of the monodromy, and the formulation above fails in general precisely because of this reason; i.e., one has to incorporate the shifts of the filtrations.
    \end{rema}

    \subsection{Affine Milnor fibration}
    We recall a preliminary fact on the vanishing cycles associated to homogeneous polynomials. Let $f \colon \CC^{m} \to \CC$ be a homogeneous polynomial of degree $d$. Then $M_{f} = f^{-1}(1)$ is called the \textit{affine Milnor fiber}. This is homotopically equivalent to the Milnor fiber at zero. The group of $d$-th roots of unity $\mu_{d}$ acts on $M_{f}$ by multiplication and this induces the monodromy transformation 
    $$T \colon M_{f} \to M_{f}, \qquad (z_{1},\ldots, z_{m}) \mapsto (e^{\frac{2\pi i}{d}} z_{1},\ldots, e^{\frac{2\pi i}{d}} z_{m}).$$
    We have the following relationship between the stalks of $\varphi_{f}$ at zero and the cohomology of the affine Milnor fiber $M_{f}$.

    \begin{prop} \label{prop:affine-milnor-fibre}
        In the set-up described above, we have an isomorphism between the stalks of the vanishing cycles and the reduced cohomology of the affine Milnor fiber, compatible with the action of the monodromy:
        $$\begin{tikzcd}
            \cH^{-i} \iota_{0}\sta \varphi_{f} \ar[r, "\simeq"] \ar[d, "T"] & \widetilde{H}^{m-1-i}(M_{f}, \QQ) \ar[d, "T"] \\
            \cH^{-i} \iota_{0}\sta \varphi_{f}  \ar[r, "\simeq"]  & \widetilde{H}^{m-1-i}(M_{f}, \QQ).
        \end{tikzcd}$$
        Here, $\iota_{0} \colon \{ 0\} \to \CC^{m}$ is the inclusion of the origin.
    \end{prop}
    Topologically, this is a classical fact (for example, see \cite{Dimca}*{\S3.1}). This isomorphism moreover holds in the level of Hodge structures. This is worked out via Alexander modules in \cite{Alexander-module}*{Corollary 7.2.4, Corollary 9.0.9} or dualizing \cite{Liu-nearby-vanishing-hodge-Alexander}*{Theorem 1.5, Proposition 1.6}.

    \begin{exam}[Simple normal crossings] \label{exam:snc-stalk-calculation}
        Let $g \colon \CC^{d} \to \CC$ given by $g = z_{1}\cdots z_{d}$. The affine Milnor fiber is isomorphic to the torus $(\CC^{\times})^{d-1}$. Multiplication by roots of unity on $(\CC^{\times})^{d-1}$ acts on the cohomology as identity since it is a part of an action of the continuous group $(S^{1})^{d-1}$. Therefore, we have
        $$ \cH^{-i} \iota_{0}\sta \varphi_{g} \simeq \begin{cases}
            \QQ(-(d-1-i))^{\oplus {d-1 \choose i}} & 0 \leq i \leq d-2 \\
            0 & \text{otherwise.}
        \end{cases} $$
        The monodromy operator $T$ on $\cH^{-i}\iota_{0}\sta \varphi_{g}$ acts as identity.
    \end{exam}

    \subsection{Singular cohomology of a hypersurface}
    We compile a series of well-known facts.

    \begin{prop} \label{prop:sing-coho-stalk}
        Let $X \subset \PP^{m+1}$ be a hypersurface of degree $d$ given by the homogeneous equation $f$. We identify $f$ with the polynomial function $f \colon \CC^{m+2} \to \CC$. Let $U = \PP^{m+1} \setminus X$ and $M_{f} = \{ f = 1\} \subset \CC^{m+2}$. Then we have a sequence of isomorphisms
        \begin{align*}
            \overline{H}^{m+i}(X, \QQ) \simeq \widetilde{H}^{m+1-i}(U, \QQ)\dual (-m-1) \simeq \left( \widetilde{H}^{m+1-i}(M_{f}, \QQ)^{\mu_{d}} \right)\dual (-m-1) \simeq \left((\cH^{-i} \iota_{0}\sta \varphi_{f})^{T} \right)\dual (-m-1).
        \end{align*}
    \end{prop}
    \begin{proof}
        The first isomorphism follows from the long exact sequence of cohomology associated to the triple $(X, \PP^{m+1}, U)$. The second isomorphism follows since $M_{f} \to U$ is a $\mu_{d}$-étale quotient where the action is given by multiplication by $d$-th roots of unity. The last isomorphism follows from Proposition \ref{prop:affine-milnor-fibre}.
    \end{proof}

    \subsection{Valuative criterion for $K$-stability}
    We briefly discuss the valuative criterion for detecting $K$-(semi)stability, developed by Fujita--Li \cite{Fujita-valuative, ChiLi-K-ss-volmin} and the equivariant version by \cite{Zhuang-equivariant} detecting $K$-polystability. We introduce relevant invariants and refer the details to \cite{Blum-Jonsson:thersholds}. Let $X$ be a Fano variety and $v$ be a divisorial valuation on $X$. For $v = r \ord_{E}$ for a divisor on $X$ given by $E \subset Y \xrightarrow{\mu} X$, we define the volume as follows:
    $$ \vol(-K_X ; v \geq \alpha) = \vol(- \mu\sta K_X - r\alpha E) = \lim_{N \to \infty} \frac{(\dim X)!}{N^{\dim X}} \dim_{\CC} \cF_{v}^{N\alpha} H^{0}(X, -NK_X), $$
    where the real-indexed filtration $\cF$ on $H^{0}(X, -NK_X)$ is defined as
    $$ \cF_{v}^{ t} H^{0}(X, -NK_X) = \{ s \in H^{0}(X, -NK_X) : v(s) \geq t\}.$$ 
    The $\delta$-invariant is given by
    $$ \delta(X) = \inf_{v} \frac{A_{X}(v)}{S_{X}(v)}, $$
    where the infimum runs over all divisorial valuations on $X$. The quantity $A_{X}(v)$ is the log discrepancy, and $S_{X}(v)$ is defined by
    $$ S_{X}(v) = \frac{1}{(-K_{X})^{n}} \int_{\alpha = 0}^{\infty} \vol(X, -K_X; v \geq \alpha) d\alpha.$$
	The valuative criterion for $K$-stability says the following:
    \begin{theo}[\cite{Fujita-valuative, ChiLi-K-ss-volmin, Blum-Jonsson:thersholds, Liu-Xu-Zhuang:finite-generation}]
        A Fano variety $X$ is $K$-stable (resp. $K$-semi-stable) if and only if $\delta (X) > 1$ (resp. $\geq 1$). 
    \end{theo}
    For the equivariant version of the valuative criterion, suppose we have a reductive algebraic group $G \subset \mathrm{Aut}(X)$ acting on a Fano variety $X$. Then we define
    $$ \delta_{G}(X) = \inf_{v} \frac{A_X(v)}{S_X(v)}$$
    where the infimum runs over \textit{$G$-invariant} divisorial valuations over $X$. Then we have
    \begin{theo}[\cite{Zhuang-equivariant}*{Corollary 4.14}] \label{theo:Zhuang-equivariant-K-polystable}
        For $X$ and $G$ as above, if $\delta_{G}(X) > 1$, then $X$ is $K$-polystable.
    \end{theo}

    \section{Singular cohomology}
    We prove Theorem \ref{theo:appending-snc}.
    \begin{proof}[Proof of Theorem \ref{theo:appending-snc}]
         We have a polynomial function $f \colon \CC^{m+2} \to \CC$ of degree $d$. We let
    $$ F(x, y) = f(x) + y_{11}\cdots y_{1d} + \ldots + y_{l1} \cdots y_{ld} \colon \CC^{m+2+ ld} \to \CC.$$
    The function $z_{1}\cdots z_{d}$ defining the snc hypersurface is denoted by $g \colon \CC^{d} \to \CC$. From Proposition \ref{prop:Thom-Sebastiani}, we have
    $$ \cH^{-i} \iota_{0}\sta \varphi_{F} \simeq \bigoplus_{i_{0}+ \ldots + i_{l} = i} \left(\cH^{-i_{0}}\iota_{0}\sta \varphi_{f}\right) \otimes\left(\cH^{-i_{1}}\iota_{0}\sta \varphi_{g}\right)\otimes \cdots \otimes \left(\cH^{-i_{l}}\iota_{0}\sta \varphi_{g}\right).$$
    Using the computation in Example \ref{exam:snc-stalk-calculation}, we get
    $$ \left( \cH^{-i} \iota_{0}\sta \varphi_{F} \right)^{T} \simeq \bigoplus_{0 \leq i_{0} \leq i} \left( \cH^{-i_{0}}\iota_{0}\sta \varphi_{f} \right)^{T} (-(d-1)l + i-i_{0})^{\oplus a_{i-i_{0}}},$$
    where the multiplicities $a_{i-i_{0}}$ are explained in Theorem \ref{theo:appending-snc}. Note that the monodromy operator acts on the stalks of $\varphi_{g}$ by identity, so computing the fixed part of the monodromy on the stalks of $\varphi_{F}$ is equivalent to computing the one for the stalks of $\varphi_{f}$. Finally, using Proposition \ref{prop:sing-coho-stalk}, we get
    \begin{align*}
        \overline{H}^{m+ld+i}(Y) &\simeq \left( (\cH^{-i}\iota_{0}\sta \varphi_{F})^{T}\right)\dual (-ld-1-m)\\
        & \simeq \bigoplus_{0 \leq i_{0} \leq i} \left( (\cH^{-i_{0}}\iota_{0}\sta \varphi_{f})^{T}\right)\dual (-m-1-l-i+i_{0})^{\oplus a_{i-i_{0}}} \\
        & \simeq \bigoplus_{0 \leq i_{0} \leq i} \overline{H}^{m+ i_{0}}(X)(-l-i+i_{0})^{\oplus a_{i-i_{0}}}
    \end{align*}
    \end{proof}

    We remark that the singular cohomology of the hypersurface
    $$ X = \{ x_{11} \cdots x_{1d} + \ldots + x_{l1} \cdots x_{ld} = 0\} \subset \PP^{ld-1}$$
    can be computed in the same way as the proof above by replacing the stalk $(\cH^{-i_{0}} \iota_{0}\sta \varphi_{f})^{T}$ by $\QQ$ if $i_{0} = 0$ and zero otherwise. This shows Theorem \ref{theo:main-theorem}(2).

    \section{Intersection cohomology}
    We compute the intersection cohomology of $X$ given by the equation $\{ F = 0  \} \subset \PP^{ld-1}$, where
    $$ F = x_{11} \cdots x_{1d} + \ldots + x_{l1} \cdots x_{ld}.$$
    Throughout this section, $X$ will always mean this hypersurface. For this, we compute the weight graded pieces of the RHM defect object $\cK_{X}$, which gives the class of $[\cK_{X}]$ in $K_{0}(\MHM(X))$, the Grothendieck group of mixed Hodge modules on $X$. Here, we do this by computing the RHM defect object $\cK_{C(X)}$ of the cone of $X$ by understanding the vanishing cycle $\varphi_{F}$. Taking the alternating sum of the cohomology gives an additive homomorphism $K_{0}(\MHM(X)) \to K_{0}(\mathrm{MHS})$ and we get the class $\sum (-1)^{i} [\cH^{i} a\lsta \cK_{X}]$, where $a \colon X \to \mathrm{pt}$ is the structure morphism. Combining it with the calculation of the singular cohomology in Theorem \ref{theo:main-theorem}(2), it is possible to compute the class of $\sum (-1)^{i} [\IH^{i}(X, \QQ)]$ in $K_{0}(\mathrm{MHS})$. The purity of intersection cohomology allows us to recover the intersection cohomology of $X$. This computation is carried out in Proposition \ref{prop:intersection-cohomology}.

    \subsection{Vanishing cycles $\varphi_{F}$}
    We first give a description for the graded pieces of $\varphi_{g}$, where $g = z_{1}\cdots z_{d} \colon \CC^{d} \to \CC$.

    \begin{lemm} \label{lemm:grW-vanishing-cycle-snc}
       Let $g\colon \CC^{d} \to \CC$ be as above. We have an isomorphism
        $$ \bigoplus_{w = 2}^{2(d-1)} \gr_{w}^{W} \varphi_{f} \simeq \bigoplus_{t = 2}^{d} \bigoplus_{|J| = t} \bigoplus_{m=0}^{t} \IC_{L_{J}}^{H}(-m-1)$$
        so that nilpotent operator $N \colon \gr_{w}^{W} \varphi_{f} \to \gr_{w-2}^{W}\varphi_{f} (-1)$ is given by the identity
            $$ \id : \IC_{L_{J}}^{H}(-m-1)  \to \IC_{L_{J}}^{H}(-m)(-1)$$
        for $m \geq 1$ on each summand and zero otherwise. Here, $L_{J}$ is the codimension $|J|$ linear subspace given by the equation $\{ z_{j} = 0 : j \in J\}$.
    \end{lemm}
    \begin{proof}
        Follows from \cite{Peters-Steenbrink:MHS}*{\S11}, which verifies this in the level of mixed $\QQ$-Hodge complexes.
    \end{proof}

    Now, we carry out the calculation of the weight graded pieces of $\varphi_{F}$. Before that, we fix some notation.

    \begin{defi} \label{defi:linear-map-V_t}
        For $\bft = (t_{1},\ldots, t_{l})$ with integers $2 \leq t_{i} \leq d$, we define
    $$ V_{\mathbf{t}}\uind{m} \coloneqq \bigoplus_{\substack{m_{1}+\ldots + m_{l} = m \\ 0 \leq m_{i} \leq t_{i}-2}} \QQ\uind{m_{1}} \otimes \ldots \QQ\uind{m_{l}}.$$
    Here $\QQ\uind{m_{i}}$ are just one-dimensional vector spaces and the superscripts are there just for the purpose of bookkeeping the indices. We define the primitive part as
     $$ V_{\bft, \prim}\uind{m} \coloneqq \ker V_{\bft}\uind{m} \to V_{\bft}\uind{m-1},$$
     where the map is given by
     $$ \sum_{\substack{1 \leq i \leq l : m_{i} \geq 1}}\sum_{m_{1},\ldots, m_{l}} \QQ\uind{m_{1}} \otimes \ldots \QQ\uind{m_{l}} \xrightarrow{1} \QQ\uind{m_{1}}\otimes \ldots \otimes \QQ\uind{m_{i}-1} \otimes \ldots \QQ\uind{m_{l}}.$$
    \end{defi}

     We give a small remark which follows from standard $\mathfrak{sl}_{2}$-representation theory.
     \begin{lemm} \label{lemm:inj-surj-of-V-uind(m)}
         $V_{\bft}\uind{m} \to V_{\bft}\uind{m-1}$ is surjective if and only if $m \leq \uround{\frac{|\bft|}{2}} - l$ and is injective if and only if $m \geq \lround{\frac{|\bft|}{2}} - l + 1$. In particular, $V_{\bft,\prim}\uind{m}$ is non-zero for $0 \leq m \leq \lround{\frac{|\bft|}{2}}-l$ where $|\bft|= t_{1}+ \ldots + t_{l}$.
     \end{lemm}

     We record the dimensions of these vector spaces.

     \begin{lemm} \label{lemm:dimension-of-Vt-prim}
         $$ \sum_{m} \dim V_{\bft, \prim}\uind{m} q^m = \tau_{<\frac{|\bft|}{2}-l+1} \left( (1-q) \prod_{i=1}^{l} \frac{q^{t_{i}-1}-1}{q-1} \right) \in \ZZ[q].$$
         The truncation $\tau_{<\frac{|\bft|}{2}-l+1}$ of a polynomial is defined as
         $$ \tau_{<\frac{|\bft|}{2}-l+1} \left( \sum_{m \geq 0} \alpha_{m} q^{m} \right) = \sum_{m < \frac{|\bft|}{2}-l+1} \alpha_{m}q^{m}.$$
     \end{lemm}
     \begin{proof}
         It is easy to see that
         $$ \sum_{m} \dim V_{\bft}\uind{m} q^{m} = \prod_{i=1}^{l} \frac{q^{t_{i}-1}-1}{q-1} \in \ZZ[q].$$
         The assertion then follows from Lemma \ref{lemm:inj-surj-of-V-uind(m)}.
     \end{proof}

     We give the description for the weight graded pieces of $\varphi_{F}$.

    \begin{lemm} \label{lemm:grW-vanishing-cycle-F}
        We have
        $$ \bigoplus_{w = 2l}^{2l(d-1)} \gr_{w}^{W} \varphi_{F} \simeq \bigoplus_{\bft} \bigoplus_{|J_{i}| = t_{i}} \bigoplus_{0 \leq m_{i} \leq t_{i}-2} \IC_{L_{J}}^{H}(-|m|-l) \simeq \bigoplus_{\bft} \bigoplus_{|J_{i}| = t_{i}} \bigoplus_{m=0}^{|\bft|-2l}\IC_{L_{J}}^{H}(-m-l) \otimes_{\QQ} V_{\bft}\uind{m}.$$
        Here, $\mathbf{t} = (t_{1},\ldots, t_{l})$ runs over integers $2 \leq t_{i} \leq d$ and $J = (J_{1},\ldots, J_{l})$ runs over all subsets $J_{i} \subset \{ 1,\ldots, d\}$ with size $t_{i}$. The linear subspace $L_{J}$ is given by $\{ z_{ij} = 0 : \text{for all } j \in J_{i}, 1 \leq i \leq l\}$. The nilpotent operator $N \colon \gr_{w}^{W} \varphi_{F} \to \gr_{w-2}^{W} \varphi_{F}(-1)$ is given by
        $$ \IC_{L_{J}}^{H}(-m-l) \otimes_{\QQ} V_{\bft}\uind{m} \to \IC_{L_{J}}^{H}(-(m-1)-l)(-1) \otimes_{\QQ} V_{\bft}\uind{m-1}$$
        for $m\geq 1$, where the morphism is given by the tensor product of the identity on $\IC_{L_{J}}^{H}$ and the linear map $V_{\bft}\uind{m} \to V_{\bft}\uind{m-1}$ as in Definition \ref{defi:linear-map-V_t}.
    \end{lemm}
    \begin{proof}
        Follows directly from Lemma \ref{lemm:grW-vanishing-cycle-snc} and Proposition \ref{prop:Thom-Sebastiani}.
    \end{proof}

    \subsection{The RHM defect object $\cK_{X}$}
    We calculate the weight graded pieces of the RHM defect object $\cK_{X}$ by calculating that of the cone $C(X)$. 
    
    \begin{prop} \label{prop:RHM-def-C(X)-graded}
        We have the following description for the weight graded pieces of $\cK_{C(X)}$:
        $$  \bigoplus_{w} \gr_{w}^{W} \cK_{C(X)}^{H} \simeq \bigoplus_{\bft} \bigoplus_{|J_{i}| = t_{i}} \bigoplus_{m \geq 0} \IC_{L_{J}}^{H}(-m-l+1) \otimes_{\QQ} V_{\bft, \prim}\uind{m}.$$
        For the right hand side, the first sum runs over all $\bft = (t_{1},\ldots, t_{l})$ such that $2 \leq t_{i} \leq d$. The second sum runs over all $J_{i} \subset \{ 1,\ldots, d\}$ such that $|J_{i}| = t_{i}$. The linear subspace $L_{J}$ is defined by $\{ z_{ij} = 0: j \in J_{i}\}$. This has codimension $|\bft|$ in $\CC^{ld}$.
    \end{prop}

    \begin{proof}
    This immediately follows from Lemma \ref{lemm:grW-vanishing-cycle-F} and Lemma \ref{lemm:defect-is-kernel-of-N}.
    \end{proof}

    \begin{prop} \label{prop:RHM-defect-object-X-graded}
        We have the following description of the RHM defect object $\cK_{X}$ of $X$:
        $$ \bigoplus_{w} \gr_{w}^{W} \cK_{X} \simeq \bigoplus_{\bft \neq (d, \ldots, d)} \bigoplus_{|J_{i}| = t_{i}} \bigoplus_{m \geq 0} \IC_{\PP(L_{J})}^{H} (-m-l+1) \otimes_{\QQ} V_{\bft, \prim}\uind{m}.$$
        Here, $\PP(L_{J})$ is the projectivization of $L_{J}$.
    \end{prop}
    \begin{proof}
        Note that $\pi \colon C(X) \setminus 0 \to X$ is a smooth morphism of relative dimension 1. Therefore, the exact sequence
        $$ 0 \to \cK_{C(X) \setminus 0} \to \QQ_{C(X) \setminus 0}^{H}[ld-1] \to \IC_{C(X)}^{H} \to 0$$
        is obtained by applying $\pi\sta$ to the sequence
        $$ 0 \to \cK_{X} \to \QQ_{X}^{H}[ld-2] \to \IC_{X}^{H} \to 0$$
        and a shift by 1. This gives the assertion, noting that the summand for $\bft = (d,\ldots, d)$ is removed since we restricted $\cK_{C(X)}$ to $C(X) \setminus 0$.
    \end{proof}

    \subsection{Intersection cohomology}
    \begin{prop} \label{prop:intersection-cohomology}
    The intersection cohomology $\IH^{\bullet}(X)$ is of Hodge--Tate type. Furthermore, we have the following formula for the generating series of the intersection Betti numbers.
    $$
        \sum \dim \IH^{2i}(X) q^{i} =\frac{q^{ld-1}-1}{q - 1}+ (q^{d-1} - (q-1)^{d-1})^{l} q^{l-1} + \sum_{\nu} {l \choose \nu_2, \nu_3, \ldots, \nu_d } {d \choose 2}^{\nu_{2}} \cdots {d \choose d}^{\nu_{d}} g_{\nu}(q) \in \ZZ[q],
    $$
    where
    $$ g_{\nu}(q) = (-1)^{\norm{\nu}} \frac{q^{ld -\norm{\nu}}-1}{q-1} \cdot \tau_{< \norm{\nu}/2} \left( (1-q)\cdot  q^{l-1} \prod_{k=2}^{d} \left(\frac{q^{k-1}-1}{q-1}\right)^{\nu_{k}}  \right).$$
    and the sum ranges over $\nu = (\nu_{2},\ldots, \nu_{d}) \in \ZZ_{\geq 0}^{d-1}$ such that $\nu_{2} + \ldots + \nu_{d} = l$ and $\nu_{d} \neq l$. Here, $\norm{\nu}  = \sum_{k=2}^{d} k \nu_{k}$.
    \end{prop}
    We first compute the contribution of the defect object in the Grothendieck group of mixed Hodge structures $K_{0}(\mathrm{MHS})$.
    \begin{lemm} \label{lemm:defect-object-Grothendieck-group}
        We have
        $$ (-1)^{ld-1} \sum_{i} (-1)^{i} [\cH^{i} a\lsta \cK_{X}] = \sum_{\nu} {l \choose \nu_{2},\ldots, \nu_{d}} {d \choose 2}^{\nu_{2}} \cdots {d \choose d}^{\nu_{d}} g_{\nu}(u) \in K_{0}(\mathrm{MHS}),$$
        where $g_{\nu}(u)$ is defined as in Proposition \ref{prop:intersection-cohomology}. Here, $u = \QQ(-1) \in K_{0}(\mathrm{MHS})$. The sum ranges over $\nu = (\nu_{2},\ldots, \nu_{d}) \in \ZZ_{\geq 0}^{d-1}$ such that $\nu_{2} + \ldots + \nu_{d} = l$ and $\nu_{d} \neq l$.
    \end{lemm}
    \begin{proof}
        The alternating sum of cohomology gives a homomorphism 
        $$\chi \colon K_{0}(\MHM(X)) \to K_{0}(\mathrm{MHS}), \qquad [\cM] \mapsto \sum (-1)^{i} [\cH^{i} a\lsta \cM ]$$
        from the Grothendieck group of mixed Hodge modules on $X$ to the Grothendieck group of mixed Hodge structures. Therefore, it is enough to calculate the image of $\bigoplus_{w} \gr_{w} \cK_{X}$ under this homomorphism. From Proposition \ref{prop:RHM-defect-object-X-graded}, we have
        $$ \chi (\cK_{X}) = \sum_{\bft \neq (d,\ldots, d)} \bigoplus_{|J_{i}| = t_{i}} \bigoplus_{m\geq 0} \chi\left(\IC_{\PP(L_{J})}^{H}(-m-l+1)\right)\cdot \dim V_{\bft, \prim}\uind{m}.$$
        Note that $\dim L_{J} = ld - |J|$, therefore $$\chi\left(\IC_{\PP(L_{J})}^{H}(-m-l+1)\right) = (-1)^{ld- |J|-1}u^{m+l-1}\frac{u^{ld -|J|}-1}{u-1}.$$
        Let $\nu_{k} = \#\{ i : t_{i} = k\}$. For $\nu = (\nu_{2},\ldots, \nu_{d})$ such that $\nu_{2} + \ldots + \nu_{d} = l$, there are
        $${l \choose \nu_{2},\ldots, \nu_{d}} {d \choose 2}^{\nu_{2}} \cdots {d \choose d}^{\nu_{d}}$$
        many possibilities for $J = (J_{1},\ldots, J_{l})$ having such value of $\nu$. Then the assertion follows by applying Lemma \ref{lemm:dimension-of-Vt-prim} and noting that
        $$ \sum_{m\geq 0} u^{m+l-1} \dim V_{\bft, \prim}\uind{m} = \tau_{<\frac{|\bft|}{2}} \left( u^{l-1} (1-u) \prod_{i=1}^{l} \frac{u^{t_{i}-1}-1}{u-1} \right) = \tau_{<\frac{\norm{\nu}}{2}} \left( u^{l-1} (1-u) \prod_{k=2}^{d} \left( \frac{u^{k-1}-1}{u-1} \right)^{\nu_{k}} \right), $$
        and $|J| = |\bft| = \norm{\nu}$.
    \end{proof}

    By calculating the contribution of the singular cohomology, we conclude this section by proving Proposition \ref{prop:intersection-cohomology}.

    \begin{proof}[Proof of Proposition \ref{prop:intersection-cohomology}]
        From the short exact sequence
        $$ 0 \to \cK_{X} \to \QQ_{X}^{H}[ld-2] \to \IC_{X}^{H} \to 0,$$
        we have
        $$ \sum_{i} (-1)^{i} [\IH^{i}(X)] = (-1)^{ld- 1} \sum_{i} (-1)^{i} [\cH^{i} a\lsta \cK_{X}] + \sum_{i} (-1)^{i} [H^{i}(X)] \in K_{0}(\mathrm{MHS}).$$
        The first object is calculated in Lemma \ref{lemm:defect-object-Grothendieck-group}. We compute the second term using Theorem \ref{theo:main-theorem}(2). We have
        \begin{align*}
            &\sum_{i} (-1)^{i} [H^{i}(X)] \\
            & = \frac{u^{ld-1}-1}{u-1} + (-1)^{ld-2} u^{l-1} (a_{0} - a_{1}u + \ldots + (-1)^{l(d-2)} a_{l(d-2)}u^{l(d-2)}) \\
            & = \frac{u^{ld-1}-1}{u-1} + (-1)^{ld-2} u^{l-1} \left( {d-1 \choose 0} - {d-1 \choose 1} u + \ldots + (-1)^{d-2} {d-1 \choose d-2} u^{d-2} \right)^{l},
        \end{align*}
        where the second equality follows from the definition of $a_{i}$ in Theorem \ref{theo:appending-snc}. The rest of the proof is a straightforward computation.
    \end{proof}
    
    \section{$K$-polystability}
    In this section, we show that the hypersurface $X_{l,d}$ in Theorem \ref{theo:main-theorem} is $K$-polystable. 

    \subsection{Automorphism group of $X$} We return to the setup where $X$ is the projective hypersurface given by the equation $x_{11} \cdots x_{1d} + \ldots + x_{l1} \cdots x_{ld} = 0$. From now on, $X$ will always be this hypersurface. Let $G_{\sigma} = \mathfrak{S}_{d}^{l} \rtimes \mathfrak{S}_{l}$ where the action $\mathfrak{S}_{l} \to \Aut(\mathfrak{S}_{d}^{l})$ is given by permuting the entries. This is naturally a subgroup of $\mathfrak{S}_{ld}$. The element $((\tau_{1},\ldots, \tau_{l}), \mu) \in G_{\sigma}$ acts on the set $\{ (i,j) : 1 \leq i \leq l, 1 \leq j \leq d\}$ by
    $$ ((\tau_{1},\ldots, \tau_{l}), \mu) (i,j) = (\mu(i), \tau_{i}(j)). $$
    We let
    $$ \widetilde{T} = \left\{ \mathrm{diag}(\lambda_{11},\ldots, \lambda_{ld}) \in \mathrm{GL}_{ld}(\CC) : \prod_{j=1}^{d} \lambda_{1j} = \ldots = \prod_{j=1}^{d} \lambda_{lj} \right\} \subset \mathrm{GL}_{ld}(\CC).$$
    We have an action $G_{\sigma} \to \mathrm{Aut}(\widetilde{T})$ by permuting the variables $\lambda_{ij}$ in the way above, and the semi-direct product $\widetilde{G} = \widetilde{T}\rtimes G_{\sigma}$ is a subgroup of $\mathrm{GL}_{ld}(\CC)$. The center $Z(\widetilde{G})$ is $\CC\sta \cdot \id$ and we let $G = \widetilde{G}/ Z(G)$. Note that $\widetilde{G}$ acts faithfully on the affine cone $C(X)$ of $X$ and $G$ acts faithfully on $X$. In fact, $G$ is actually the automorphism group of $X$. We note that $G$ is reductive since there is a compact Zariski dense subgroup given by a semidirect product of a finite group and a real torus.

    \begin{prop} \label{prop:automorphism-of-X}
        Let $G$ and $X$ as above. Then $\Aut(X) = G$.
    \end{prop}
    \begin{proof}
        The Grothendieck--Lefschetz theorem \cite{Grothendieck-SGA2}*{Exposé XII, Corollaire 3.6}, tells us that all the automorphisms of $X$ are linear. By Proposition \ref{prop:RHM-defect-object-X-graded}, the support of $\gr_{2(l-1)}^{W} \QQ_{X}^{H}[\dim X]$ is the union of all coordinate points in $\PP^{ld-1}$. This locus should be fixed by an automorphism of $X$. Therefore, the automorphisms are compositions of permutation matrices and diagonal matrices. It is easy to show that among those, the matrices in $\mathrm{GL}_{ld}(\CC)$ fixing $x_{11} \cdots x_{1d} + \ldots + x_{l1} \cdots x_{ld}$ modulo scalar multiples are exactly the elements in $\widetilde{G}$. The assertion for $\Aut(X)$ follows since $G = \widetilde{G}/\CC\sta \cdot \id$.
    \end{proof}

    We fix some notations to encode the torus action. Consider the free abelian group $\ZZ^{ld}$ with basis elements $\{ e_{ij} : 1 \leq i \leq l, 1 \leq j \leq d\}$. Let $\eps_{i} = e_{i1} + \ldots + e_{id}$ for $1 \leq i \leq l$ and consider the subgroup $L \subset \ZZ^{ld}$ generated by $\eps_{1} - \eps_{2},\ldots, \eps_{1} - \eps_{l}$. It is easy to see that $L$ is saturated in $\ZZ^{ld}$; hence the quotient $M \coloneq \ZZ^{ld}/L$ is a free abelian group. We view $e_{ij}$ and $\eps = \eps_{1} = \ldots = \eps_{l}$ as elements in $M$. Note that the torus $\Spec \CC[M]$ is naturally isomorphic to $\widetilde{T}$ where the character $e_{ij}$ corresponds to the coordinate $\lambda_{ij}$ of $\widetilde{T}$. We let $\sigma\dual \subset M$ be the cone spanned by the vectors $e_{ij} \in M$. The dual here is a purely symbolic notation motivated by toric geometry. There is a well-defined morphism $\deg \colon M \to \ZZ$ by assigning $\deg(e_{ij}) = 1$ for all $e_{ij}$. This corresponds to the one-parameter subgroup given by the diagonal $\CC\sta \cdot \id$ of $\widetilde{T}$.

    \begin{rema}
        Some of the notations and the arguments in this section are motivated by those in toric geometry and in the theory of $T$-varieties \cite{AIPSV-T-var}. However, we decided not to use the terminology therein for the readers who are unfamiliar with $T$-varieties and to keep the necessary ingredients as minimal as possible.
    \end{rema}

    \subsection{Structure of the cone $\sigma\dual$} We describe the structure of the cone $\sigma\dual$.
    \begin{prop} \label{prop:structure-of-sigma-dual}
        Let $\sigma\dual$ be as in the previous section. Then the number of facets of $\sigma\dual$ is $d^{l}$. For each $\vec{j} = (j_{1},\ldots, j_{l}) \in \{1,\ldots, d\}^{l}$, we have a corresponding facet $\tau_{\vec{j}}$ generated by $\{ e_{ij} : 1 \leq i \leq l , j \neq j_{i}\}$.
    \end{prop}
    \begin{proof}
        It is easy to check that $\{ \eps \} \cup \{ e_{ij} : 1 \leq i \leq l , j \neq j_{i}\}$ spans $M$ integrally. Since $\dim M = ld - l + 1$, this is a $\ZZ$-basis for $M$. Therefore, there is a linear functional $\varphi_{\vec{j}}$ such that $\varphi_{\vec{j}}(\eps) = 1$ and $\varphi_{\vec{j}}(e_{ij}) = 0$ for $1 \leq i \leq l$ and $j \neq j_{i}$. This implies that $\varphi_{\vec{j}}(e_{ij_{i}}) = \varphi_{\vec{j}}(\eps - \sum_{j \neq j_{i}}e_{ij}) = 1$. Therefore, $\varphi_{\vec{j}}$ pairs non-negatively with $\sigma\dual$ and zero on $\tau_{\vec{j}}$. This shows that $\tau_{\vec{j}}$ is a facet of $\sigma\dual$. In order to show that these are all the facets, it is enough to show that $ \sigma\dual = \bigcap_{\vec{j}} \{ \varphi_{\vec{j}} \geq 0\}$. The inclusion $\subseteq$ is obvious. For the other inclusion, let $\sum_{i,j} a_{ij} e_{ij}$ be an element in the right-hand side. From $\varphi_{\vec{j}}(\sum a_{ij} e_{ij}) = \sum_{i=1}^{l} a_{ij_{i}}$, we get $\sum_{i=1}^{l} \min_{1 \leq j \leq d} \{ a_{ij}\} \geq 0$. By adding or subtracting $\eps_{i} - \eps_{i'}$ suitably, we can assume that $\min_{1 \leq j \leq d} \{ a_{ij} \} \geq 0$ for all $1 \leq i \leq l$. This shows that $\sum_{ij} a_{ij} e_{ij} \in \sigma\dual$.
    \end{proof}

    From the description of the facets, we have a chamber decomposition of $\sigma\dual$ as follows. For each $\vec{j} \in \{ 1, \ldots, d\}^{l}$, we get a chamber $\widetilde{\tau}_{\vec{j}} \coloneqq \mathrm{span}_{\RR \geq 0}(\tau_{\vec{j}}, \eps)$. Then we get a piecewise linear function $\varphi$ given by $\varphi(u) = \varphi_{\vec{j}}(u)$ for $u \in \widetilde{\tau}_{\vec{j}}$. In order to check that this is well-defined, i.e., agrees on the intersection of various $\widetilde{\tau}_{\vec{j}}$'s, one can verify that for $u = \sum a_{ij} e_{ij}$, we have
    $$ \varphi(u) = \min_{1 \leq j \leq d} \{ a_{1j} \} + \ldots + \min_{1 \leq j \leq d} \{ a_{lj} \}.$$
    It is clear that $\varphi$ takes integer values in the lattice points of $\sigma\dual \cap M$ and takes the value zero exactly on the boundary of the cone $\sigma\dual$.
    \begin{defi}
        We define $\widetilde{\tau}_{\vec{j}}$ and $\varphi$ as above.
    \end{defi}
    We prove two small lemmas that will be used later.
    \begin{lemm} \label{lemm:decomposition-of-character}
        For $u \in \sigma\dual$, there are unique integers $b_{ij} \in \ZZ_{\geq 0}$ for $1 \leq i \leq l, 1 \leq j \leq d$ satisfying
        $$ u = \varphi(u) \eps + \sum_{\substack{1 \leq i \leq l \\ 1 \leq j \leq d}} b_{ij} e_{ij}.$$
    \end{lemm}
    \begin{proof}
        We can assume that $\varphi(u) = 0$ since $u - \varphi(u) \eps \in \sigma\dual$ and $\varphi(u - \varphi(u)\eps) = 0$. Then the assertion is clear.
    \end{proof}

    \begin{lemm} \label{lemm:number-of-lattice-points}
        For $u \in \sigma\dual \cap M$ such that $\deg(u) = N(l-1)d$, we have $0 \leq \varphi(u) \leq (l-1)N$. There is a constant $C(l,d)$ depending only on $l, d$ such that
        $$ \left| \# \Big\{ u \in \sigma\dual \cap M : \varphi(u) =m, \deg(u) = Nd(l-1)\Big\} - d^{ld-1}\frac{(N(l-1)-m)^{l(d-1)-1}}{(l(d-1)-1)!} \right| \leq C(l,d) N^{l(d-1)-2}$$
        for all $0 \leq m \leq (l-1)N$, if $N$ is large enough.
    \end{lemm}
    \begin{proof}
        The first assertion is clear by applying the degree to Lemma \ref{lemm:decomposition-of-character}. Let $u \in \sigma\dual \cap M$ such that $\varphi(u) = m$ and $\deg(u) = Nd(l-1)$. Without loss of generality, suppose that $\vec{j} = (d,d,\ldots, d)$ and $u \in \widetilde{\tau}_{\vec{j}}$. The $u$ can be uniquely expressed as
        $$ u = m \eps + \sum_{\substack{1 \leq i \leq l \\ 1 \leq j \leq d-1}} a_{ij} \eps_{ij} .$$
        Since $\deg (u) = Nd(l-1)$, we have $\sum_{1 \leq i \leq l, 1\leq j \leq d-1} a_{ij} = d(N(l-1)-m)$. The number of $(d-1)l$ non-negative integers that sum up to $d(N(l-1)-m)$ is ${ d(N(l-1)-m) \choose l(d-1)-1}$. Then the assertion easily follows since they are $d^{l}$ facets of $\sigma\dual$ and the contribution from the overlap of $\widetilde{\tau}_{\vec{j}}$ is uniformly bounded by a constant multiple of $N^{l(d-1)-2}$.
    \end{proof}

    \subsection{$G$-invariant valuations on $X$}
    We characterize $G$-invariant valuations on $X$. Before that, we set up some notation. Let $Z = \{ x_{11}\cdots x_{1d} = \ldots = x_{l1}\cdots x_{ld} = 0\}$ and let $\widetilde{X}$ be the blow-up along $Z$, i.e.,
    $$ \widetilde{X} = \Big\{ ([x_{ij}], [z_{k}])\in \PP^{ld-1} \times \PP^{l-1} : \substack{x_{11}\cdots x_{1d} + \ldots + x_{l1}\cdots x_{ld} = 0\\z_{1} + \ldots + z_{l} = 0\\ z_{i}x_{i'1}\cdots x_{i'd} = z_{i'}x_{i1}\cdots x_{id} \text{ for all } 1 \leq i,i' \leq l} \Big\}.$$
    We denote the projection to the second coordinate by $\pi \colon \widetilde{X} \to \PP^{l-2}$, where $\PP^{l-2}$ is viewed as a hyperplane in $\PP^{l-1}$ given by the equation $z_{1}+ \ldots + z_{l} = 0$. The coordinate ring of the cone $C(X)$ over $X$ is $R = \CC[x_{ij}]/(F)$ and it has a natural grading by $M$ coming from the action of the torus $\widetilde{T}$. We denote the homogeneous parts by
    $$ R = \bigoplus_{u \in \sigma\dual \cap M} \Gamma_{u} = \bigoplus_{m\geq 0} R_{m},$$
    where $R_{m} = \bigoplus_{\deg(u) = m} \Gamma_{u}$. We first describe the homogeneous parts.
    \begin{lemm} \label{lemm:homogeneous-part}
        Let $\Gamma_{u}$ as above. Let $u = \varphi(u)\eps + \sum b_{ij} e_{ij}$ as in Lemma \ref{lemm:decomposition-of-character}. Then we have
        $$ \Gamma_{u} = H^{0}(\PP^{l-2}, \cO_{\PP^{l-2}}(\varphi(u)) )\cdot \prod_{i,j} x_{ij}^{b_{ij}},$$
        where the identification is induced by the inclusion $\CC[z_{1},\ldots, z_{l}]/(z_{1} +\ldots + z_{l}) \hookrightarrow \CC[x_{ij}]/(F)$ given by $z_{i} \mapsto x_{i1}\cdots x_{id}$.
    \end{lemm}
    \begin{proof}
        Clear from the description.
    \end{proof}
    
    In the following proposition, we show that the $G$-invariant valuations on $X$ are completely characterized by valuations on $\PP^{l-2}$.
    \begin{prop}\label{prop:valuations-on-X}
        We have a one-to-one correspondence between $G$-invariant valuations $v$ on $X$, $\widetilde{G}$-invariant valuations $\widetilde{v}$ on $C(X)$ centered at the origin up to scaling, and $\mathfrak{S}_{l}$-invariant valuations $v_{0}$ on $\PP^{l-2}$. The $\mathfrak{S}_{l}$-action on $\PP^{l-2}$ is given by the standard representation. Moreover, we have the following properties:
        \begin{enumerate}
            \item $v$ and $v_{0}$ are restrictions of $\widetilde{v}$ under the following field inclusions:
            $$ \CC(\PP^{l-2}) \simeq \mathrm{Frac}(R)^{\widetilde{T}} \subset \CC(X) \simeq \mathrm{Frac}(R)^{\CC\sta \cdot \id} \subset \mathrm{Frac}(R). $$
            Here, $\mathrm{Frac}(R)^{\widetilde{T}}$ and $\mathrm{Frac}(R)^{\CC\sta \cdot \id}$ are the fixed parts by the action of $\widetilde{T}$. The first inclusion is identified with the one induced by the projection $\pi\colon \widetilde{X} \to \PP^{l-2}$.
            \item If $v$ is a $G$-invariant valuation on $X$, then the center $c_{X}(v)$ is not contained in the union of hyperplanes $\{x_{ij} =0\}$. In particular, it is not contained in $Z = \{ x_{11}\cdots x_{1d} = \ldots = x_{l1}\cdots x_{ld} = 0\}$.
            \item If $v_{0}$ is a $\mathfrak{S}_{l}$-invariant valuation on $\PP^{l-2}$, then the center $c_{\PP^{l-2}}(v_{0})$ is not contained in the union of hyperplanes $\bigcup_{i=1}^{l} \{ z_{i} = 0\}$.
            \item If $\widetilde{v}$ is a $\widetilde{G}$-invariant valuation on $C(X)$ centered at the origin and $f = \sum_{u \in \sigma\dual \cap M} f_{u} = \sum_{m \geq 0} f_{m} \in R$ is the decomposition into homogeneous parts, we have
            $$ \widetilde{v}(f) = \min_{u} \widetilde{v}(f_{u}) = \min_{m} \widetilde{v}(f_{m}).$$
            \item If $v_{0}$ is a $\mathfrak{S}_{l}$-invariant valuation on $\PP^{l-2}$, then the corresponding valuation $\widetilde{v}$ is determined by property (4) and the following description on the homogeneous parts:
            $$ \widetilde{v}(f_{u}) = \deg(u) + v_{0}(f_{u}/ \prod_{ij} x_{ij}^{a_{ij}}),$$
            for $f_{u} \in \Gamma_{u}$, where $u = \sum a_{ij} e_{ij}$.
            \item Similarly, if $v$ is a $G$-invariant valuation on $X$, then the $\widetilde{v}$ is determined by property (4) and the following analogous description:
            $$ \widetilde{v}(f_{m}) = m + v(f/x_{11}^{m}),$$
            for $f_{m} \in R_{m}$.
            \item $v$ is divisorial if and only if $v_{0}$ is divisorial.
        \end{enumerate}
    \end{prop}
    \begin{proof}
        The one-to-one correspondence follows from (1), (5), and (6) where the scaling for $\widetilde{v}$ is determined by fixing $\widetilde{v}(x_{11}) = 1$. Part (1) is clear. We note that the inclusion $\CC(\PP^{l-2}) \hookrightarrow \CC(X)$ is given by $z_{i}/z_{i'} \mapsto x_{i1}\cdots x_{id}/x_{i'1}\cdots x_{i'd}$.
        
        For (2), consider a $G$-invariant valuation $v$ on $X$. Then the center $c_{X}(v)$ intersects one of the open charts $\{ x_{ij} \neq 0\}$. Since the finite group $G_{\sigma} \subset G$ acts transitively on the variables $x_{ij}$, the center $c_{X}(v)$ should intersect the intersection of all of these charts. This shows (2). The third item can be shown similarly.
        
        We prove (4). Let
        $$ f = f_{u_{1}} + \ldots + f_{u_{m}}$$
        where $f_{u_{i}} \in \Gamma_{u_{i}}$. It is clear that $\widetilde{v}(f) \geq \min_{1 \leq t \leq m} \widetilde{v}(f_{u_{t}})$. Consider $w \in \Hom_{\ZZ}(M, \ZZ)$ such that all the numbers $\langle w , u_{i}\rangle$ are different. The action of the one parameter subgroup of $\widetilde{T}$ corresponding to $w$ on $f$ is given by
        $$ \lambda_{w} \cdot f = \lambda^{\langle w , u_{i} \rangle} f_{u_{1}} + \ldots + \lambda^{\langle w, u_{t} \rangle} f_{u_{t}}.$$
        By taking various values of $\lambda$, one can show that each $f_{u_{i}}$ can be expressed as a linear combination of $\lambda_{w} \cdot f$. By $\widetilde{G}$-invariance assumption, we have $\widetilde{v}(f) = \widetilde{v}(\lambda_{w} \cdot f)$. This shows that $\widetilde{v}(f_{u_{i}}) \geq \widetilde{v}(f)$ for $1 \leq i \leq t$ which verifies the other inequality. The second equality can be shown similarly.
        
        We show (5). Note that $\widetilde{v}$ is well-defined since $v_{0}(x_{i1}\cdots x_{id}/x_{i'1}\cdots x_{i'd}) = 0$ due to $\mathfrak{S}_{l}$-invariance, by considering the involution switching $i$ and $i'$. In particular, the formula in (5) does not depend on the choice of the representation $u = \sum a_{ij} e_{ij}$ in $M$. Since $\widetilde{v}$ is centered at the maximal ideal, we have $\widetilde{v}(x_{11}) > 0$. Scaling $\widetilde{v}$, we assume that $\widetilde{v}(x_{11}) = 1$. Since the finite group $G_{\sigma}$ acts on the variables transitively, we have $\widetilde{v}(x_{ij}) = 1$ for all $i,j$ as well. This shows that $\widetilde{v}$ is indeed a valuation. The invariance of $\widetilde{v}$ under the torus action $\widetilde{T}$ holds by construction. The invariance under the action of $G_{\sigma} = (\mathfrak{S}_{d}^{l}) \rtimes \mathfrak{S}_{l}$ also follows since $G_{\sigma}$ does not change the degree of $u$, and $\mathfrak{S}_{d}^{l}$ acts trivially on $\CC(\PP^{l-2})$. Item (6) can be proved similarly.
        
        We finally show item (7). Note that the morphism $\pi \colon \widetilde{X} \to \PP^{l-2}$ is a product fibration over the locus $U = \PP^{l-2} \setminus \bigcup_{i=1}^{l} \{ z_{i} \neq 0\}$. Indeed, it is not hard to see that $\pi^{-1}(\widetilde{X}) \simeq X_{\Sigma} \times U$ over $U$, where $X_{\Sigma}$ is the toric variety $\{ x_{11}\cdots x_{1d} = \ldots = x_{l1}  \cdots x_{ld}\} \subset \PP^{ld-1}$. Suppose that $v_0$ is divisorial. Since the center of $v_{0}$ intersects $U$, there is a proper birational morphism $\psi \colon V \to U$ such that $v_{0} = r \ord_{E}$ for a divisor $E \subset V$. Then we have a base change diagram
        $$ \begin{tikzcd}
            E \times X_{\Sigma} \subset V \ar[r] \ar[d] & \pi^{-1}(U) \ar[d] \\ E \subset V \ar[r] & U,
        \end{tikzcd}$$
        and $r \ord_{E \times X_{\Sigma}}$ is a $G$-invariant valuation on $\widetilde{X}$ whose restriction to $\CC(\PP^{l-2})$ is $v_{0}$. Hence, $v = r \ord_{E \times X_{\Sigma}}$ is divisorial.
        For the converse, let $v$ be a $G$-invariant divisorial valuation on $X$. By (2), the center $c_{X}(v)$ is not contained in the union of hyperplanes $\{ x_{ij} = 0\}$, so $c_{\widetilde{X}}(v)$ intersects the locus
        $$ W = \{ ([x_{ij}], [z_{k}]) \in \widetilde{X} : x_{ij} \neq 0, z_{k} \neq 0\}.$$
        One can see that $W \to U$ is a principal $T$-bundle over $U$. By \cite{Kollar-Mori}*{Lemma 2.45}, $v$ is obtained by successive blow-ups along $T$-invariant centers. Since $W \to U$ is a principal $T$-bundle, the sequence of blow-ups descend to blow-ups of $U$ and we obtain a divisor over $U$ computing the valuation $v_{0}$.
    \end{proof}

    From the description of the relation between $v$ and $v_{0}$, it is easy to compute the log discrepancy of $v$.

    \begin{prop} \label{prop:log-discrepancy}
        Let $v$ be a $G$-invariant divisorial valuation on $X$ and $v_{0}$ be the corresponding valuation on $\PP^{l-2}$ as in Proposition \ref{prop:valuations-on-X}. Then $A_{X}(v) = A_{\PP^{l-2}}(v_{0})$.
    \end{prop}
    \begin{proof}
        The center $c_{X}(v)$ avoids $Z$ by Proposition \ref{prop:valuations-on-X}.(2). Therefore, $A_{X}(v) = A_{\widetilde{X}}(v)$. It is clear from the description in the proof of Prop \ref{prop:valuations-on-X}.(7) that $A_{\widetilde{X}}(v) = A_{\PP^{l-2}}(v_{0})$.
    \end{proof}

    \subsection{Calculation of the volume}
    For a $G$-invariant divisorial valuation $v$ on $X$, we calculate the volume $\vol(-K_X ; v\geq \alpha)$.

    \begin{prop} \label{prop:volume-of-v}
        If $v$ is a $G$-invariant divisorial valuation on $X$ and $v_{0}$ is the $\mathfrak{S}_{l}$-invariant divisorial valuation on $\PP^{l-2}$ as in Proposition \ref{prop:valuations-on-X}, we have
        $$ \vol(-K_X ; v \geq \alpha) = \frac{(ld-2)! \cdot d^{ld-1}}{(l-2)! (ld-l-1)!}\int_{0}^{l-1} (l-1-x)^{l(d-1)-1} \vol_{\PP^{l-2}}(x\cdot\cO_{\PP^{l-2}}(1); v_{0}\geq \alpha) dx.$$
        The expression $\vol(x \cdot \cO_{\PP^{l-2}}(1); v_{0}\geq \alpha)$ on the right hand side is understood as the volume of the $\RR$-divisor $x\cdot \mu\sta \cO_{\PP^{l-2}}(1) - \alpha r E$, where $v_{0} = r \ord_{E}$ for a divisor $E$ on some birational model $\mu \colon Y \to \PP^{l-2}$.
    \end{prop}

    We show a lemma first.
    \begin{lemm} \label{lemm:bound-for-volume}
        Let $v_0$ be a divisorial valuation on $\PP^{l-2}$ and $\alpha > 0$ be a real number. For $\eps > 0$, we have
        $$ \left| \frac{\dim_\CC \cF_{v_{0}}^{N\alpha} H^{0}(\PP^{l-2}, \cO_{\PP^{l-2}}(i)) }{N^{l-2}} - \frac{1}{(l-2)!} \vol_{\PP^{l-2}}\left( \frac{i}{N}\cdot \cO_{\PP^{l-2}}(1), v_{0} \geq \alpha \right) \right|\leq \eps$$
        for all $0 \leq i \leq (l-1)N$, if $N$ is large enough.
    \end{lemm}
    \begin{proof}
        Note that $f(t) =  \vol_{\PP^{l-2}}(t \cdot \cO_{\PP^{l-2}}(1), v_0 \geq \alpha)$ is a continuous increasing function on $[0, l-1]$. Let $0 = t_0 < \ldots < t_{m} = l-1$ such that $|f(t_{k+1}) - f(t_{k})| < \eps / 2$. For $N$ big enough, we have
        $$ \left| \frac{\dim_{\CC} \cF_{v_{0}}^{N\alpha} H^{0}(\PP^{l-2}, \cO_{\PP^{l-2}}(\lround{t_{k}N}) }{N^{l-2}} - \frac{1}{(l-2)!} \vol_{\PP^{l-2}}(t_{k}\cdot \cO_{\PP^{l-2}}(1), v_{0}\geq \alpha) \right| <\eps/2, $$
        for all $k = 0, \ldots, m$. Then for $0 \leq i \leq (l-1)N$, pick $k$ such that $\lround{Nt_{k}} \leq i < \lround{Nt_{k+1}}$. We see that
        \begin{align*}
            & \left| \frac{\dim_\CC \cF_{v_{0}}^{N\alpha} H^{0}(\PP^{l-2}, \cO_{\PP^{l-2}}(i)) }{N^{l-2}} \right| \geq \left| \frac{\dim_{\CC} \cF_{v_{0}}^{N\alpha} H^{0}(\PP^{l-2}, \cO_{\PP^{l-2}}(\lround{t_{k}N}) }{N^{l-2}} \right| \\
            & \geq \frac{1}{(l-2)!} \vol_{\PP^{l-2}}(t_{k}\cdot \cO_{\PP^{l-2}}(1), v_{0}\geq \alpha) - \frac{\eps}{2} \geq \frac{1}{(l-2)!}\vol_{\PP^{l-2}}\left(\frac{i}{N}\cdot \cO_{\PP^{l-2}}(1), v_{0}\geq \alpha\right) - \left(1 + \frac{1}{(l-2)!} \right)\frac{\eps}{2}.
        \end{align*}
        A similar inequality for the other direction using $Nt_{k+1}$ deduces the lemma.
    \end{proof}

    We show Proposition \ref{prop:volume-of-v}.

    \begin{proof}
        We recall the notation $R = \CC[x_{ij}]/(F)$ the homogeneous coordinate ring of $X$, and the decomposition by homogeneous parts 
        $$R = \bigoplus_{m \geq 0} R_{m} = \bigoplus_{u \in \sigma\dual \cap M}\Gamma_{u} .$$ 
        Note that $\omega_X = \cO_X(-ld+d)$ by the adjunction formula, hence
        $$ H^{0}(X, -NK_X) = R_{N(l-1)d} =\bigoplus_{\deg(u) = N(l-1)d} \Gamma_{u}.$$
        Let $v$ be a divisorial valuation on $X$. Under this identification, we have
        $$ \cF_v^{N\alpha} H^{0}(X, -NK_X) \simeq \cF_{\widetilde{v}}^{N\alpha + N(l-1)d} R_{N(l-1)d},$$
        where $\widetilde{v}$ is the valuation on $R$ introduced in Proposition \ref{prop:valuations-on-X}. Similarly, under the identification in Lemma \ref{lemm:homogeneous-part}, we have
        $$ \Gamma_{u} \cap \cF_{\widetilde{v}}^{N\alpha + N(l-1)d} R_{N(l-1)d} \simeq \cF_{v_{0}}^{N\alpha} H^{0}(\PP^{l-2}, \cO_{\PP^{l-2}}(\varphi(u))),$$
        for $\deg(u) = N(l-1)d$. By Proposition \ref{prop:valuations-on-X}(4), we see that the filtration $\cF_{\widetilde{v}}^{\bullet}$ on $R$ is compatible with the decomposition by homogeneous parts. In particular, we have
        $$ \dim_{\CC} \cF_{v}^{N\alpha} H^{0}(X, -NK_X) = \sum_{\deg(u) = Nd(l-1)} \dim \cF_{v_{0}}^{N\alpha} H^{0}(\PP^{l-2}, \cO_{\PP^{l-2}}(\varphi(u))).$$
        Denote by
        $$ N_{m} = \# \{ u \in \sigma\dual \cap M : \deg(u) = (l-1)Nd, \varphi(u) = m\}.$$
        We get
        $$ \frac{\dim_{\CC} \cF_{v}^{N\alpha} H^{0}(X, -NK_X)}{N^{\dim X}} = \frac{1}{N^{ld-2}} \sum_{m=0}^{(l-1)N} N_{m} \dim_{\CC} \cF_{v_{0}}^{N\alpha} H^{0}(\PP^{l-2}, \cO_{\PP^{l-2}}(m)).$$
        Note that by Lemma \ref{lemm:number-of-lattice-points}, we have
        \begin{align*}
            & \frac{1}{N^{ld-2}} \sum_{m=0}^{(l-1)N} \left| N_{m} - d^{ld-1} \frac{(N(l-1) -m)^{l(d-1)-1}}{(l(d-1)-1)!} \right| \cdot \dim_{\CC} \cF_{v_{0}}^{N\alpha} H^{0}(\PP^{l-2}, \cO_{\PP^{l-2}}(m)) \\
            &\lesssim \frac{N \cdot N^{l(d-1)-2} \cdot N^{l-2}}{N^{ld-2}} = \frac{1}{N} \xrightarrow{N \to \infty}0.
        \end{align*}
        Fix $\eps >0$. Then for $N$ big enough, we have
        \begin{align*}
            & \frac{1}{N^{ld-2}} \sum_{m=0}^{(l-1)N} (N(l-1) - m)^{l(d-1)-1} \left| \dim_{\CC} \cF_{v_{0}}^{N\alpha} H^{0}(\PP^{l-2} ,\cO_{\PP^{l-2}}(m)) - \frac{N^{l-2}}{(l-2)!} \vol_{\PP^{l-2}} \left( \frac{m}{N} \cdot \cO_{\PP^{l-2}}(1) ; v_{0} \geq \alpha \right) \right| \\
            & \leq \frac{\eps N^{l-2}}{N^{ld-2}} \sum_{m=0}^{(l-1)N} (N(l-1) - m)^{l(d-1)-1} \leq (l-1)^{l(d-1)}\eps \frac{N^{l-2} \cdot N^{l(d-1)}}{N^{ld-2}} = (l-1)^{l(d-1)}\eps.
        \end{align*}
        This shows that
        \begingroup\makeatletter\def\f@size{8}\check@mathfonts
        $$ \left| \frac{\dim_{\CC} \cF_{v}^{N\alpha} H^{0}(X, -NK_X)}{N^{\dim X}} - \frac{d^{ld-1} N^{l-2}}{(l(d-1)-1)! (l-2)! N^{ld-2}} \sum_{m=0}^{(l-1)N} \vol_{\PP^{l-2}} \left( \frac{m}{N}\cdot \cO_{\PP^{l-2}}(1); v_{0} \geq \alpha\right) \cdot (N(l-1)-m)^{l(d-1)-1} \right| \xrightarrow{N \to \infty}0.$$
        \endgroup
        Therefore, we have
        \begin{align*}
            & \vol(-K_X ; v \geq \alpha) \\
            & = \frac{(ld-2)! \cdot d^{ld-1}}{(l(d-1)-1)! (l-2)!} \lim_{N \to \infty} \sum_{m=0}^{(l-1)N} \vol_{\PP^{l-2}}\left( \frac{m}{N} \cdot \cO_{\PP^{l-2}}(1); v_{0} \geq \alpha \right) \left( l-1 - \frac{m}{N} \right)^{l(d-1)-1} \frac{1}{N} \\
            & = \frac{(ld-2)! \cdot d^{ld-1}}{(l(d-1)-1)! (l-2)!} \int_{0}^{l-1} (l-1-x)^{l(d-1)-1} \vol_{\PP^{l-2}}\left(x\cdot \cO_{\PP^{l-2}}(1); v_{0} \geq \alpha \right) dx
        \end{align*}
    \end{proof}

    \subsection{$K$-polystability of $X$}
    We show that $X$ is $K$-polystable.

    \begin{prop} \label{prop:S-invariant-of-X}
        If $v$ is a $G$-invariant divisorial valuation over $X$, and $v_{0}$ is the corresponding valuation on $\PP^{l-2}$, as in Proposition \ref{prop:valuations-on-X}, we have the following equality of $S$-invariants:
        $$ S_{X}(v) = \frac{l-1}{ld-1} S_{\PP^{l-2}}(v_{0}).$$
    \end{prop}

    Note that this immediately proves that $X$ is $K$-polystable.

    \begin{proof}[Proof of Theorem \ref{theo:K-polystable}]
        Combining Proposition \ref{prop:log-discrepancy} and \ref{prop:S-invariant-of-X}, we see that for every $G$-invariant divisorial valuation on $X$, we have
        $$ \frac{A_{X}(v)}{S_{X}(v)} = \frac{ld-1}{l-1} \cdot \frac{A_{\PP^{l-2}}(v_{0})}{S_{\PP^{l-2}}(v_{0})} \geq \frac{ld-1}{l-1} >1.$$
        Here, $v_{0}$ is the corresponding valuation on $\PP^{l-2}$ as in Proposition \ref{prop:valuations-on-X}. The last inequality comes from the $K$-semistability of projective spaces. The assertion follows from Theorem \ref{theo:Zhuang-equivariant-K-polystable}.
    \end{proof}

    We finally show Proposition \ref{prop:S-invariant-of-X}.

    \begin{proof}[Proof of Proposition \ref{prop:S-invariant-of-X}]
        Note that $\vol(X, -K_{X}) = d^{ld-1} (l-1)^{ld-2}.$ Hence,
    $$ S_{X}(v) = {ld-2 \choose l-1} \frac{1}{(l-1)^{ld-3}} \int_{\alpha = 0}^{\infty} \int_{x = 0}^{l-1}  (l-1-x)^{l(d-1)-1} \vol_{\PP^{l-2}}(x\cdot \cO_{\PP^{l-2}}(1), v_{0}\geq \alpha) dxd\alpha.$$
    Note that
    $$ \int_{\alpha = 0}^{\infty} \vol_{\PP^{l-2}}(x\cdot\cO_{\PP^{l-2}}(1), v_{0}\geq \alpha) d\alpha  = x^{l-1} \int_{\alpha = 0}^{\infty} \vol_{\PP^{l-2}} (\cO_{\PP^{l-2}}(1), v_{0}\geq \alpha) d\alpha = \frac{x^{l-1}}{l-1} S_{\PP^{l-2}}(v_{0}).$$
    Therefore, we have
    \begin{align*}
        S_{X}(v_{X}) & = S_{\PP^{l-2}}(v_{0})\cdot {ld-2 \choose l-1} \frac{1}{(l-1)^{ld-2}} \int_{x=0}^{l-1} (l-1-x)^{l(d-1)-1} x^{l-1} dx\\
        & = S_{\PP^{l-2}}(v_{0}) (l-1) {ld-2 \choose l-1} \int_{0}^{1} (1-x)^{l(d-1)-1} x^{l-1} dx \\
        & = S_{\PP^{l-2}}(v_{0}) (l-1) {ld-2 \choose l-1} \frac{(l(d-1)-1)! (l-1)!}{(ld-1)!} = S_{\PP^{l-2}}(v_{0})\frac{l-1}{ld-1} .
    \end{align*}
    The last integral is commonly known as the Euler integral of the first kind.
    \end{proof}
    
    \subsection{Local structure of the $K$-moduli space}
    Here, we analyze the local structure of the $K$-moduli space near the hypersurface $X$. The space of first order deformations of $X$ is $\mathbb{T}^{1} = \Ext(\Omega_{X}^{1}, \cO_{X})$ and the deformation of $X$ is unobstructed. By the Luna étale slice theorem for algebraic stacks \cite{Alper-Hall-Rydh:Luna}, the local structure of the $K$-moduli space $\cM^{K}$ is given by the following Cartesian square:
    $$ \begin{tikzcd}
        {[}\mathbb{T}^{1}/\!\Aut(X){]} \ar[r] \ar[d] & \cM^{K} \ar[d] \\
        \mathbb{T}^{1}/\!/\!\Aut(X) \ar[r]& M^{K},
    \end{tikzcd} $$
    where the horizontal maps are étale.

    We describe $\bb{T}^{1}$ and the action of $\Aut(X)$ on it. The following proposition says that the formula describing the infinitesimal deformations of smooth hypersurfaces works almost verbatim. The computation is standard but we decided to include it for readers convenience.

    \begin{prop} \label{prop:deformation-space-of-X}
        We have an isomorphism
        $$ \bb{T}^{1} \simeq \frac{\CC[x_{ij}]_{d}}{\mathrm{span}\langle x^{\eps_{i} - e_{ij} + e_{i'j'}} : 1 \leq i, i' \leq l, 1\leq j,j'\leq d \rangle}.$$
        The torus $T \subset \Aut(X)$ acts as
        $$ \lambda \cdot g(x) \mapsto \frac{g(\lambda \cdot x)}{\lambda_{11}\cdots \lambda_{1d}}$$
        for $\lambda = (\lambda_{11},\ldots, \lambda_{ld}) \in T$ and $g \in \bb{T}^{1}$.
    \end{prop}

    \begin{proof}
        From the conormal sequence
        $$ 0 \to \cO_{X}(-X) \to \Omega_{\PP^{ld-1}}^{1}|_{X} \to \Omega_{X}^{1} \to 0,$$
        we get an exact sequence
        $$0 \to \Hom_{\cO_{X}}(\Omega_{X}^{1} ,\cO_{X}) \to H^{0}(\cT_{\PP^{ld-1}}|_{X}) \xrightarrow{\delta} H^{0}(\cO_{X}(X)) \to \bb{T}^{1} \to H^{1}(\cT_{\PP^{ld-1}}|_{X}).$$
        Consider the exact sequence
        $$ 0 \to \cT_{\PP^{ld-1}}(-X) \to \cT_{\PP^{ld-1}} \to \cT_{\PP^{ld-1}}|_{X} \to 0$$
        and note that $\cT_{\PP^{ld-1}}(-d)$ has no cohomology since $ld + 1 - d \geq 2$. This follows from Borel--Weil--Bott calculation, for example, see \cite{Borisov-C-Perry}*{\S A.1}. Therefore, we have an isomorphism $H^{i}(\cT_{\PP^{ld-1}}) \to H^{i}(\cT_{\PP^{ld-1}}|_{X})$. Therefore, we see that
        $$ H^{1}(\cT_{\PP^{ld-1}}|_{X}) = 0, \quad H^{0}(\cT_{\PP^{ld-1}}|_{X}) \simeq \mathrm{span}\langle x_{ij} \de_{i'j'} \rangle\Big/\sum_{i,j} x_{ij}\de_{ij}, \quad H^{0}(\cO_{X}(X)) = R_{d}/\langle F \rangle $$
        The map $\delta$ is given by
        $$ x_{i'j'} \de_{ij} \mapsto x_{i'j'}\de_{ij} F = x^{\eps_{i} - e_{ij} + e_{i'j'}}.$$
        Then the assertion for $\bb{T}^{1}$ follows. The torus action on $\bb{T}^{1}$ can be seen from the torus action on $H^{0}(X, \cO_{X}(X))$.
    \end{proof}

    We note that the family $\cX$ over $\bb{T}^{1}$
    $$ \cX = \left\{ (x, G) \in \PP^{ld-1} \times \bb{T}^{1} : (F + G)(x) = 0 \right\} \to \bb{T}^{1}$$
    for a chosen inclusion $\mathbb{T}^{1} \hookrightarrow \CC[x_{ij}]_{d}$ restricts to the miniversal deformation of $X$ at the origin.

    We now analyze the stability of the action of $\Aut(X)$ on $\bb{T}^1$. Note that $\Aut(X) = T \rtimes G_{\sigma}$ where $G_{\sigma} = \mathfrak{S}_{d}^{l} \rtimes \mathfrak{S}_{l}$ is a finite group. Therefore, the stability of a vector $v \in \bb{T}^{1}$ is completely characterized by the action of the torus $T$ on $\bb{T}^{1}$. When a torus acts diagonally on a vector space, the characterization of polystable points is well known. For example, see \cite{Dolgachev-Invariant-theory} or \cite{Popov:invariant-theory}*{Proposition 6.15}.

    \begin{prop}
        Let $V$ be a vector space and $T$ be a torus that acts diagonally on $V$. Let $M = \Hom(T, \CC\sta)$ be the group of characters and consider the decomposition
        $$ V = \bigoplus_{u \in M} V_{u}.$$
        For $v \in V$, we write $v = \sum_{u \in M} v_{u}$ such that $v_{u} \in V_{u}$. We consider the polytope
        $$ \Phi(v) = \mathrm{conv} \{ u : v_{u} \neq 0\},$$
        the convex hull of the support of the weights of $v$. The vector $v$ is polystable if and only if $0 \in \mathrm{Relint} \Phi(v)$.
    \end{prop}

    We describe the eigenspace decomposition of $\bb{T}^{1}$ with respect to the action of the torus $T \subset \Aut(X)$. Note that the character group of $T$ can be identified canonically with $M_{0} = \{ u \in M : \deg(u) = 0\}$.

    \begin{lemm} \label{lemm:defi-of-Phi}
        Let $\Psi = \{ u \in M_{0} : u + \eps \in \sigma\dual \}$ and
        $$\Phi = \Psi \setminus \{ e_{ij} - e_{i'j'} : 1 \leq i,i' \leq l, 1 \leq j,j' \leq d \}.$$
        We have $\bb{T}^{1} = \bigoplus_{u \in \Phi} \bb{T}^{1}_{u}$ with $\dim \bb{T}^{1}_{u} = 1$.
    \end{lemm}
    \begin{proof}
        From Proposition \ref{prop:deformation-space-of-X}, some collection of degree $d$ monomials from a basis of $\bb{T}^{1}$ diagonalizing the torus action. The degree $d$ monomials have weights in $\Psi$ due to the factor of $\lambda_{11}^{-1}\cdots \lambda_{1d}^{-1}$ in the action of $T$ on $\bb{T}^{1}$. Note that every eigenspace of $\CC[x_{ij}]_{d}$ has dimension 1 except $\mathrm{span} \langle x^{\eps_{1}},\ldots, x^{\eps_{l}}\rangle$. However, these vectors become trivial on $\bb{T}^{1}$. Note that the support of the eigenspace decomposition $\Phi$ is also clear from the description of $\bb{T}^{1}$ in Proposition \ref{prop:deformation-space-of-X}.
    \end{proof}

    For example, it is easy to show that the hypersurface given by the equation $f(x_{11},\ldots, x_{td}) + x_{t+1,1}\cdots x_{t+1,d} + \ldots + x_{l1} \cdots x_{ld} = 0 $ is $K$-polystable for $f$ general, just by looking at the support of the monomials appearing in this polynomial.

    \begin{proof}
        Let $M^{\leq t} \subset M$ be the subgroup generated by $e_{ij}$ for $1 \leq i \leq t$ and $1 \leq j \leq d$. Define $M_{0}^{\leq t} = M_{0} \cap M^{\leq t}$ and $\Phi_{0}^{\leq t} = \Phi \cap M_{0}^{\leq t}$, where $\Phi$ is as in Lemma \ref{lemm:defi-of-Phi}. The deformation to the hypersurface
        $$f(x_{11},\ldots, x_{td}) + x_{t+1,1}\cdots x_{t+1,d} + \ldots + x_{l1} \cdots x_{ld} = 0$$
        is given by the general linear combination of monomials in $\Phi_{0}^{\leq t}$. It is easy to see that $\Phi_{0}^{\leq t}$ spans $M_{0}^{\leq t}$ and $\sum_{u \in \Phi_{0}^{\leq t}} u = 0$. This tells us that $0 \in \mathrm{Relint} (\mathrm{conv} \Phi_{0}^{\leq t})$. Therefore
        $$ \{ f(x_{11},\ldots, x_{td}) + x_{t+1,1}\cdots x_{t+1,d} + \ldots + x_{l1} \cdots x_{ld} = 0 \} $$
        is $K$-polystable for $f$ general.
    \end{proof}

    \section{GIT stability}
    We finish the paper by proving Proposition \ref{prop:appending-snc-preserves-GIT}. The proof is an easy consequence of Luna's criterion and the Hilbert--Mumford numerical criterion. We first state these criteria for readers' convenience.

    \begin{prop}[Hilbert--Mumford numerical criterion \cite{Mumford:GIT}]
        Let $w = (w_{0},\ldots, w_{n}) \in \QQ^{n+1}$ be a rational root system. For a polynomial $f = \sum_{\alpha \in \ZZ_{\geq 0}^{n+1}} c_{\alpha} x^{\alpha} \in \CC[x_{0},\ldots, x_{n}]$, the weight $\mathrm{wt}_{w}(f)$ is given by the following formula:
        $$ \mathrm{wt}_{w}(f) = \min \{ w_{0}\alpha_{0} + \ldots + w_{n}\alpha_{n} : c_{\alpha} \neq 0 \}.$$
        Let $X$ be the hypersurface in $\PP^{n}$ given by a homogeneous equation $f(x_{0},\ldots,x_{n}) = 0$. Then $X$ is GIT stable (resp.\! GIT semi-stable) if and only if for all $w \in \QQ^{n+1}$ and all linear change of coordinates $g \in SL(n+1)$, we have $\mathrm{wt}_{w}(f \circ g) < \frac{d}{n+1} \sum_{i=0}^{n} w_{i}$ (resp.\! $\leq$). Furthermore, $X$ is GIT-polystable if it is GIT-semistable and if all the monomials in $f \circ g$ have the same weight if the inequality in the criterion is an equality.
    \end{prop}

    Luna's criterion allows us to efficiently test the Hilbert--Mumford criterion when the hypersurface has extra symmetry. As in the set-up above, assume that $f$ the fixed by a reductive subgroup $H$ of $SL(n+1)$. Then in order to check GIT semi(resp.\! poly)-stability, it is enough to check the numerical criterion only for $g \in C_{SL(n+1)}(H)$, the centralizer of $H$, rather than the whole group $SL(n+1)$ (see \cite{Popov:invariant-theory}*{pg. 221-222}).

    \begin{proof}[Proof of Proposition \ref{prop:appending-snc-preserves-GIT}]
        Let $f(x) \in \CC[x_{0},\ldots, x_{m+1}]$ be a polynomial such that $X = \{ f(x) =0\} \subset \PP^{m+1}$ is GIT-semistable, and $Y$ be the corresponding hypersurface given by the equation $F(y,x) = y_{1}\cdots y_{d} + f(x) = 0$ in $\PP^{m+1+d}$. Note that the autormorphism group of $Y$ contains a torus $\{ (t_{1},\ldots, t_{d})\in (\CC\sta)^{d} : t_{1}\cdots t_{d} = 1\}$, where the action is given by the multiplication on the $y$-coordinates. The centralizer of this torus in $SL(d + m + 2)$ is
        $$ H =  \left\{ \begin{bmatrix} \Lambda & 0 \\ 0 & g \end{bmatrix} : \Lambda \in (\CC\sta)^{d}, g \in GL(m+2), \det(\Lambda) \det(g) = 1 \right\}.$$
        By applying Luna's criterion, it is enough to check the Hilbert--Mumford numerical criterion after conjugating $F$ with the element in $H$. Let $\widetilde{g} = \mathrm{diag} (\Lambda, g) \in H$ and pick a weight system $(w_{1},\ldots, w_{d}, v_{0}, \ldots, v_{n}) \in \QQ^{m+d+2}$. Then we have
        $$ \mathrm{wt}_{(w,v)}(F \circ \widetilde{g}) = \min \left\{ \sum_{i=1}^{d} w_{i} , \mathrm{wt}_{v}(f \circ g) \right\} \leq \min \left\{ \sum_{i=1}^{d} w_{i} , \frac{d}{m+2} \sum_{j=0}^{m+1} v_{j}\right\} \leq \frac{d}{m+d+2} \left( \sum_{i=1}^{d} w_{i} + \sum_{j=1}^{m+1} v_{j} \right).$$
        The first inequality follows from the GIT semistability of $X$. The above inequality is an equality if and only if
        $$ \sum_{i=1}^{d} w_{i} = \frac{d}{m+2} \sum_{j=0}^{m+1} v_{j} = \mathrm{wt}_{v}(f\circ g).$$
        Therefore, the analogous statement for polystability also holds.
    \end{proof}

    \textbf{Acknowledgements} The author would like to thank 
    Daniel Brogan, Qianyu Chen, Kristin DeVleming, Brad Dirks, Masafumi Hattori, Mattias Jonsson, Riku Kurama, Lauren\c{t}iu Maxim, Mircea Musta\c{t}\u{a}, Sung Gi Park, Christian Schnell, Saket Shah, and Chenyang Xu for many helpful conversations.
    
	\bibliographystyle{alpha}
	\bibliography{Reference}

\end{document}